\documentclass[12pt, hyphens]{article}
\usepackage[a4paper, 
            top=1.0in, 
            bottom=1.0in, 
            left=1.0in, 
            right=1.0in,
            hmarginratio=1:1,
            bindingoffset=0.0in
]
{geometry}
\usepackage[utf8]{inputenc}
\usepackage[T1]{fontenc}
\usepackage{lmodern}
\usepackage[english]{babel}
\usepackage{indentfirst}
\usepackage{amsmath,amssymb,mathtools}
\usepackage{mathrsfs} 
\usepackage{microtype}
\usepackage{euscript}
\usepackage[dvipsnames, svgnames]{xcolor} 
\usepackage{bm} 
\usepackage[misc]{ifsym}
\usepackage[normalem]{ulem} 
\usepackage{setspace}
\usepackage{caption}
\usepackage{subcaption}
\usepackage{graphicx}
\usepackage{xparse}
\usepackage{etoolbox}
\usepackage{booktabs}       % professional-quality tables
\usepackage{nicefrac}       % compact symbols for 1/2, etc.
\usepackage{stmaryrd}
\usepackage{algorithm}
\usepackage{algpseudocode}

\usepackage[shortlabels]{enumitem} 
\newcounter{assumpitem}
\usepackage[numbers, sort&compress]{natbib}
\usepackage[unicode=true, pdfnewwindow]{hyperref}
\hypersetup{pdfencoding=auto, psdextra, colorlinks=false, pdfborder={0 0 1}, linktoc=page}
\usepackage[numbered]{bookmark}% always after hyperref
\usepackage{amsthm}
\theoremstyle{plain}
\newtheorem{theorem}{Theorem}[section]

\newtheorem{lemma}[theorem]{Lemma}

\newtheorem{proposition}[theorem]{Proposition}

\newtheorem{claim}[theorem]{Claim}
\theoremstyle{definition}
\newtheorem{definition}[theorem]{Definition}
\newtheorem{remark}[theorem]{Remark}
\newtheorem{example}[theorem]{Example}

\newtheorem{assumptions}[theorem]{Assumptions}

\numberwithin{equation}{section}
\makeatletter
\if@titlepage%
\else
    \def\abstract@style{heading} % heading | inline
    \def\abstract@size{\small}
    
    \newcommand{\setabstractstyle}[1]{%
      \def\abstract@style{#1}%
    }
    
    \newcommand{\setabstractsize}[1]{%
      \def\abstract@size{#1}%
    }
    
    \let\orig@abstract\abstract
    \let\endorig@abstract\endabstract
    
    \renewenvironment{abstract}{%
      \if@twocolumn
        \orig@abstract%
      \else
        \abstract@size%
        \edef\temp@a{inline}%
        \ifx\abstract@style\temp@a%
          \quotation
          \noindent\textbf{\abstractname.}\ %
          \ignorespaces
        \else
          \begin{center}%
          {\bfseries\abstractname\vspace{-.5em}\vspace{\z@}}%
          \end{center}%
          \quotation
          \ignorespaces
        \fi
      \fi
    }{%
      \if@twocolumn
        \endorig@abstract
      \else
        \endquotation % both inline and heading
      \fi
    }
\fi
\makeatother
\setabstractstyle{inline} % "heading" default for article class, also inline is an option
\setabstractsize{\normalsize} % \small is default for article
\renewcommand{\proofname}{Proof}% 
\makeatletter
\renewenvironment{proof}[1][\proofname]{%
  \par\pushQED{\qed}\normalfont
  \topsep6\p@\@plus6\p@\relax % topsep by .. can stretch (@plus) by ..
  \trivlist
  \item[\hskip\labelsep\bfseries #1\@addpunct{.}]\ignorespaces
}{%
  \popQED\endtrivlist\@endpefalse
}
\makeatother
\usepackage{tocloft}
\makeatletter
\newif\ifinappendixbookmark
\inappendixbookmarkfalse

\let\oldappendix\appendix
\renewcommand{\appendix}{%
  \oldappendix
  \inappendixbookmarktrue 
}
\makeatother

\newif\ifdisableappendixbookmark
\disableappendixbookmarkfalse

\makeatletter
\newcommand*{\AddAppendixPrefixInBookmarks}{%
  \ifinappendixbookmark
    \ifnum\bookmarkget{level}=1
      \ifdisableappendixbookmark
      \else
        \preto\bookmark@text{Appendix\space}%
      \fi
    \fi
  \fi
}
\bookmarksetup{
  addtohook=\AddAppendixPrefixInBookmarks,
}
\makeatother
\makeatletter
\renewcommand\Hy@numberline[1]{#1.\space}%
\makeatother

\makeatletter
\newif\if@appendixsec
\@appendixsecfalse
\apptocmd{\appendix}{%
  \@appendixsectrue
}{}{}

\renewcommand{\@seccntformat}[1]{%
  \if@appendixsec
    \ifstrequal{#1}{section}{%
      Appendix~\thesection.\hskip.5em%\quad% \hskip.5em = \enspace
    }{%
      \csname the#1\endcsname.\hskip.5em%\quad%
    }%
  \else
    \csname the#1\endcsname.\hskip.5em%\quad%
  \fi
}
\makeatother
\makeatletter
\newif\ifinappendixtoc
\inappendixtocfalse

\let\oldappendixtoc\appendix
\renewcommand{\appendix}{%
  \oldappendixtoc
  \addtocontents{toc}{\protect\inappendixtoctrue}%
}
\makeatother

\makeatletter
\let\oldnumberline\numberline
\renewcommand{\numberline}[1]{%
  \ifinappendixtoc
    \oldnumberline{\vphantom{[]}#1.\hskip.5em}%\quad}%
  \else
    \oldnumberline{#1.\vphantom{[]}}%
  \fi
}
\makeatother

\makeatletter
\let\oldcontentsline\contentsline

\renewcommand{\contentsline}[4]{%
  \ifinappendixtoc
    \ifstrequal{#1}{section}{%
      \oldcontentsline{#1}{\vphantom{[]}Appendix~#2}{\vphantom{[]}#3}{#4}%
    }{%
      \oldcontentsline{#1}{\vphantom{[]}#2}{\vphantom{[]}#3}{#4}%
    }%
  \else
    \oldcontentsline{#1}{\vphantom{[]}#2}{\vphantom{[]}#3}{#4}%
  \fi
}
\makeatother
\makeatletter
\AtBeginDocument{%
    \pretocmd{\NAT@hyper@}{\vphantom{[]}}{}{}% \NAT@num 
    \def\def@NAT@last@yr#1{%
     \protected@edef\NAT@last@yr{%
      #1%
      \noexpand\mbox{%
       \noexpand\hyper@natlinkstart{\@citeb\@extra@b@citeb}%
       {\noexpand\citenumfont{\NAT@num\vphantom{[]}}}%
       \noexpand\hyper@natlinkend
      }%
     }%
    }%
}
\makeatother

\let\oldbibliographystyle\bibliographystyle
\renewcommand{\bibliographystyle}[1]{%
    \addtocontents{toc}{\protect\inappendixtocfalse}
    \disableappendixbookmarktrue
    \oldbibliographystyle{#1}%
}

\AtEndDocument{
\addcontentsline{toc}{section}{\refname}
}

\let\LTXlabel\label
\makeatletter
\renewcommand{\footnote}[2][\empty]{%
  \nolinebreak%
  \refstepcounter{footnote}% <-- FIX: step once, correctly
  \global\edef\sfootnote@arabic{\arabic{footnote}}%
  \global\protected@edef\sfootnote@the{\thefootnote}%
  \ltx@ifpackageloaded{hyperref}{%
    \ifHy@hyperfootnotes%
      \refstepcounter{Hfootnote}% <-- FIX: keep in sync
      \global\let\Hy@saved@currentHref\@currentHref%
      \hyper@makecurrent{Hfootnote}%
      \global\let\Hy@footnote@currentHref\@currentHref%
      \global\let\@currentHref\Hy@saved@currentHref%
    \fi%
  }{}%
  \xdef\sfootnote@opt{#1}%
  \ifx\sfootnote@opt\empty%
    \footnotetext{\LTXlabel{fnr:\sfootnote@arabic}#2}%
  \else%
    \ltx@ifpackageloaded{hyperref}{%
      \footnotetext[#1]{\phantomsection\LTXlabel{fnr:\sfootnote@arabic}#2\vphantom{Xg}}%
    }{%
      \footnotetext[#1]{\LTXlabel{fnr:\sfootnote@arabic}#2}%
    }%
  \fi%
  \ltx@ifpackageloaded{hyperref}{%
    \ifHy@hyperfootnotes%
      \hbox{\@textsuperscript{\,\normalfont\ref{fnr:\sfootnote@arabic}\vphantom{Xg}}}% I put \,
    \else%
      \hbox{\@textsuperscript{\normalfont\ref*{fnr:\sfootnote@arabic}}}%
    \fi%
  }{%
    \hbox{\@textsuperscript{\normalfont\ref{fnr:\sfootnote@arabic}}}%
  }%
  \;%
  \ignorespaces%
}
\makeatother

\makeatletter
\newif\if@maketitle@nohyper
\let\orig@maketitle\maketitle
\renewcommand{\maketitle}{%
  \@maketitle@nohypertrue
  \orig@maketitle
  \@maketitle@nohyperfalse
}
\makeatother

\makeatletter
\renewcommand{\footnotemark}[1][]{%
  \leavevmode
  \ifhmode\edef\@x@sf{\the\spacefactor}\nobreak\fi
  \def\@tempa{#1}%
  \ifx\@tempa\@empty
    \stepcounter{footnote}%
    \protected@xdef\@thefnmark{\thefootnote}%
  \else
    \begingroup
      \c@footnote #1\relax
      \unrestored@protected@xdef\@thefnmark{\thefootnote}%
    \endgroup
  \fi
  \ltx@ifpackageloaded{hyperref}{%
  \ifHy@hyperfootnotes
    \if@maketitle@nohyper
    \else
      \refstepcounter{Hfootnote}%
      \global\let\Hy@saved@currentHref\@currentHref
      \hyper@makecurrent{Hfootnote}%
      \global\let\Hy@footnote@currentHref\@currentHref
      \global\let\@currentHref\Hy@saved@currentHref
    \fi
  \fi
  }{}%
  \textsuperscript{%
    \normalfont
    \ltx@ifpackageloaded{hyperref}{%
      \ifHy@hyperfootnotes
        \if@maketitle@nohyper
          \@thefnmark
        \else
          \,\hyper@linkstart{link}{\Hy@footnote@currentHref}%
          \@thefnmark\vphantom{Xg}%
          \hyper@linkend
        \fi
      \else
        \@thefnmark
      \fi
    }{%
       \@thefnmark
    }%
  }%
  \ifhmode\spacefactor\@x@sf\fi
  \;\ignorespaces
}
\makeatother

\newcommand{\myref}[1]{%
  \mbox{\ref{#1}\vphantom{()}}%
}
\makeatletter
\DeclareRobustCommand{\myhyperref}[2]{%
  \mbox{%
    \hyperref[#2]{%
     \ref*{#1}\;\textup{\tagform@{\ref*{#2}}}\,%
    }\vphantom{()}%
  }%
}
\makeatletter
\DeclareRobustCommand{\myeqhyperrefrange}[4]{%
  \mbox{%
    \hyperref[#2]{%
      \ref*{#1}\;\textup{\tagform@{\ref*{#3}}}\,--\,\textup{\tagform@{\ref*{#4}}}\,%
    }\vphantom{()}%
  }%
}
\makeatother

\newcommand{\myEqref}[1]{\mbox{\hyperref[#1]{(\ref*{#1})}\vphantom{()}}}
\makeatletter
\NewDocumentCommand{\runinsectionstar}{o m}{%
  \IfNoValueTF{#1}{%
    \def\runin@size{\Large}%
  }{%
    \def\runin@size{#1}%
    \if\relax\detokenize{#1}\relax
      \def\runin@size{\Large}%
    \fi
  }%
  \addcontentsline{toc}{section}{#2}%
  \@startsection{section}%
    {1}%
    {\z@}%
    {-3.5ex \@plus -1ex \@minus -.2ex}%
    {-1em}%
    {\normalfont\runin@size\bfseries}*%
    {#2.}%
}
\makeatother
\makeatletter
\NewDocumentCommand{\runinsubsectionstar}{o m}{%
  \IfNoValueTF{#1}{%
    \def\runin@size{\large}%
  }{%
    \def\runin@size{#1}%
    \if\relax\detokenize{#1}\relax
      \def\runin@size{\large}%
    \fi
  }%
  \addcontentsline{toc}{subsection}{#2}%
  \@startsection{subsection}%
    {2}%
    {\z@}%
    {-3.25ex \@plus -1ex \@minus -.2ex}%
    {-1em}%
    {\normalfont\runin@size\bfseries}*%
    {#2.}%
}
\makeatother

\makeatletter
\NewDocumentCommand{\runinsubsection}{o m}{%
  \IfNoValueTF{#1}{%
    \def\runin@size{\large}%
  }{%
    \def\runin@size{#1}%
    \if\relax\detokenize{#1}\relax
      \def\runin@size{\large}%
    \fi
  }%
  \@startsection{subsection}%
    {2}%
    {\z@}%
    {-3.25ex \@plus -1ex \@minus -.2ex}%
    {-1em}%
    {\normalfont\runin@size\bfseries}%
    {#2.}%
}
\makeatother
\DeclareMathOperator*{\argmin}{arg\,min}
\DeclareMathOperator{\jac}{Jac}
\DeclareMathOperator*{\graph}{gph}
\DeclareMathOperator*{\minimize}{minimize}
\DeclareMathOperator*{\cl}{cl}
\DeclareMathOperator*{\interior}{int}
\DeclareMathOperator*{\conv}{conv\!}
\DeclareMathOperator*{\kkt}{Mult}

\DeclareMathOperator*{\dist}{dist}
\newcommand{\isomorph}{%
  \mathrel{\lower0.2ex\hbox{$\cong$}}
}
\providecommand{\transp}[1]{\ensuremath {#1}^{\!\mathsf{T}}}
\providecommand{\norm}[1]{\ensuremath \lVert{#1}\rVert}
\providecommand{\scalarp}[1]{\ensuremath \langle{#1}\rangle}
\providecommand{\metricd}[2]{\ensuremath \norm{{#1}-{#2}}}
\providecommand{\indicator}{\ensuremath \delta}

\newcommand{\vx}{\ensuremath x}
\newcommand{\vy}{\ensuremath y}
\newcommand{\vz}{\ensuremath z}
\newcommand{\vu}{\ensuremath u}
\newcommand{\vv}{\ensuremath v}
\newcommand{\vV}{\ensuremath V}
\newcommand{\vw}{\ensuremath w}

\newcommand{\vt}{\ensuremath t}
\newcommand{\tht}{\ensuremath \theta}
\newcommand{\lbd}{\ensuremath \lambda}

\newcommand{\param}{\ensuremath \tht}
\newcommand{\pvar}{\ensuremath \vx}
\newcommand{\dvar}{\lbd}
\newcommand{\dvarr}{\mu}
\newcommand{\st}{\mathrm{subject\text{ }to}}
\newcommand{\R}{\ensuremath \mathbb{R}}
\newcommand{\N}{\ensuremath \mathbb{N}}

\newcommand{\Rell}{\ensuremath \R^{\ell}}
\newcommand{\Rm}{\ensuremath \R^m}
\newcommand{\Rn}{\ensuremath \R^n}
\newcommand{\Rp}{\ensuremath \R^p}
\newcommand{\Rq}{\ensuremath \R^q}
\newcommand{\Rr}{\ensuremath \R^r}
\newcommand{\Rs}{\ensuremath \R^s}
\newcommand{\Lcal}{\ensuremath \mathcal{L}}
\newcommand{\Ucal}{\ensuremath \mathcal{U}}
\newcommand{\Bcal}{\ensuremath \mathcal{B}}

\newcommand{\Ncal}{\ensuremath \mathcal{N}}
\newcommand{\Ocal}{\ensuremath \mathcal{O}}

\newcommand{\solpd}{\ensuremath \mathrm{Sol_{pd}}}
\newcommand{\solvpd}{\ensuremath \mathrm{solver_{pd}}}

\newcommand{\solp}{\ensuremath \mathrm{Sol_{p}}}
\newcommand{\solvp}{\ensuremath \mathrm{solver_{p}}}
\newcommand{\consv}{\ensuremath D}
\newcommand{\clarke}{\ensuremath \partial^{c}}
\newcommand{\clarkep}[1]{\ensuremath \partial^{c}_{#1}}
\newcommand{\infny}{\ensuremath \infty}

\usepackage{academicons}
\title{The adjoint state method for parametric definable optimization without  smoothness or uniqueness}
\author{
J\'er\^ome Bolte\thanks{Toulouse School of Economics, Toulouse Capitole University, Toulouse, France \\(\texttt{cheik.traore@tse-fr.eu})}
\and
Edouard Pauwels\footnotemark[1]
\and   
Cheik Traor\'e$\,^{\textrm{\Letter}}\,$\footnotemark[1]
}
\date{\today}
\begin{document}
\maketitle
\begin{abstract}
We establish that nonconvex definable parametric optimization problems with possibly nonsmooth objectives, inequality constraints, conic constraint systems,  and non-unique primal and dual solutions admit an adjoint state formula under a mere qualification condition. The adjoint construction yields a selection of a conservative field for the value function, providing a computable first-order object without requiring differentiation of the solution mapping. Through examples, we show that even in smooth problems, the formal adjoint construction fails without conservativity or definability, illustrating the relevance of these concepts to grasp theoretical aspects of the method. This work provides a tool which can be directly combined with existing primal-dual solvers for a wide range of parametric optimization problems.
\end{abstract}

[1.5ex]
% \tableofcontents
%%%%%%%%%%%%%%%%%%%%%%%%%%%%%%%%%%%%%%%%%%%%%%%%%%%%%%%%%%%%%%%%%%%%%%%%%%%%%%%%%%%%%
%%%%%%%%%%%%%%%%%%%%%%%%%%%%%%%%%%%%% Content starts %%%%%%%%%%%%%%%%%%%%%%%%%%%%%%%%
%%%%%%%%%%%%%%%%%%%%%%%%%%%%%%%%%%%%%%%%%%%%%%%%%%%%%%%%%%%%%%%%%%%%%%%%%%%%%%%%%%%%
\section{Introduction}

\runinsubsectionstar[\normalsize]{Parametric optimization and the adjoint method} Parametric optimization problems are central in optimization \cite{bonnans1998optimization}, control \cite{lemos2025parametric}, and engineering \cite{castillo2006optimal}. They also arise naturally in modern machine learning, where one seeks to tune hyperparameters, architectures, regularization levels, or embedded optimization layers through the optimization of an induced value function; see \cite{pedregosa2016hyperparameter, bai2019deep}. In  these situations, the quantity of interest can often be recast  in the parametric form 
\begin{align}\label{eq:closedcone}
    \inf_{\pvar \in \Rn} \left\{F(\pvar, \param): C(\pvar, \param) \in K\right\},
\end{align}
where the parameter $\param$ ranges in \(\Rq\), \(F\colon \Rn \times \Rq \to \R\) is locally Lipschitz,  the constraint function \(C\colon \Rn \times \Rq \to \R^m\) is continuously differentiable and $K$ is a closed convex cone in $\R^m$.  For simplicity of exposition, we consider momentarily the following simpler instance 
\begin{align}\label{eq:fvdef}
    f(\param) \coloneq \inf_{\pvar \in \Rn} \left\{F(\pvar, \param): G(\pvar, \param) \leq 0,  H(\pvar, \param) = 0 \right\},
\end{align}
 where the constraint functions \(G\colon \Rn \times \Rq \to \R^m\) and \(H\colon \Rn \times \Rq \to \R^p\) are continuously differentiable, and the inequality is understood coordinate-wise.  A central difficulty in optimizing the value function $f$ is to obtain useful first-order information at a reasonable computational cost. In principle, if one denotes by 
\begin{align*}
x^*(\param) \in \argmin_{x \in \mathbb{R}^n}
\left\{F(x,\param): G(\pvar, \param) \leq 0,  H(\pvar, \param) = 0\right\},
\end{align*}
a selection of minimizers, the value writes
$
f(\param)=F(x^*(\param),\param)$. Thus, 
if the mapping $\param \mapsto x^*(\param)$ were differentiable, the chain rule would give
\[
\nabla f(\param)
= \transp{\nabla_{\!x} F(x^*(\param),\param)}
\frac{d x^*(\param)}{d\param}
+ \nabla_{\!\param} F(x^*(\param),\param).
\]
However, this approach raises both computational and conceptual difficulties. Computationally, differentiating through a solver is generally expensive, while conceptually it relies on strong regularity assumptions, such as those required by the implicit function theorem.

The adjoint state method, introduced by Céa \cite{cea1986conception}, avoids computing the derivative of the solution mapping and does not resort to convex duality \cite{rockafellar1974conjugate, ekeland1974analyse}. Denoting $(x^*(\param),\lambda^*(\param),\mu^*(\param))$ a solution of the KKT system of \eqref{eq:fvdef}, the adjoint formula writes (see \cite{cea1986conception})
\begin{equation}
    \label{form:asm}
\nabla f(\param)
=
\nabla_{\!\param} F(x^*(\param),\param)
+
\jac_{\param}\!\transp{G(x^*(\param),\param)}\dvar^*(\param)
+
\jac_{\param}\!\transp{H(x^*(\param),\param)}\dvarr^*(\param).
\end{equation}
Remarkably, at a negligible additional cost---see Remark~\myref{rmk:costdualvar}---associated with computing Lagrange multipliers, this formula does not involve the derivative of the primal-dual solution mapping. Yet, the classical validity analysis for the adjoint formula requires solutions of the KKT system to be unique and differentiable. To our knowledge, this requirement has not been relaxed; see the discussion in the related work section and Example~\myref{ex:failclarke}.
%%%%%%%%%%%%%%%%%%%%%%%%%%%%%%%%%%%%%%%%%%%%%%
\runinsubsectionstar[\normalsize]{The adjoint formula in the  definable world}
The objective of this work is to establish simple and versatile guarantees ensuring that the right-hand side of the adjoint formula \eqref{form:asm} is a gradient-type quantity for the value function $f$.

For nonlinear programming problems \eqref{eq:fvdef}, our results have an appealing form: assuming that the Mangasarian–Fromovitz constraint qualification holds and that all the mappings $F,G,H$ belong to a common o-minimal structure (for example, are semialgebraic), the formula yields a conservative field, and in particular a gradient almost everywhere. We actually show that a similar result holds in the much more general setting \eqref{eq:closedcone}, under Robinson's qualification and o-minimality.

Our approach  relies on two key ingredients: definability and conservative calculus. Definability provides a natural framework that encompasses most optimization and machine learning models \cite{bolte2021conservative, davis2020stochastic}, while offering powerful stability and regularity properties that are essential for our analysis. Conservative calculus, on the other hand, supplies generalized differential objects suited to nonsmooth calculus and has already led to strong applications in optimization and learning; see \cite{bolte2021conservative,davis2020stochastic,pauwels2023conservative}. The use of conservative fields for the value function is essential here: the formal application of the adjoint formula in \eqref{form:asm} produces spurious outputs, even for smooth data $F,G,H$ resulting in a smooth value function $f$; see the very simple Example~\myref{ex:failclarke}. This is very common in nonsmooth differential calculus. Nonetheless, the outputs of the adjoint formula \eqref{form:asm} form a conservative field; in particular, we obtain a gradient almost everywhere. 

To conclude, we emphasize that, in the absence of definability, the adjoint formula may fail to carry variational meaning for the value function, even for locally Lipschitz, path differentiable and (point-wise) differentiable objective function with smooth constraints; see Section~\myref{sect:failure}.

\runinsubsectionstar[\normalsize]{The adjoint formula and applications}  Although primarily theoretical, this work provides a rigorous foundation for the practical use of the adjoint state method in settings where smoothness or uniqueness of solutions cannot be guaranteed. Such situations arise in a wide range of applications, including robust optimization \cite{ben2002robust}, uncertainty quantification \cite{allaire2015review}, optimal design of structures and systems \cite{allaire2015review}, parametric decomposition \cite{gauvin1978method,demiguel2008decomposition}, PDE-constrained optimization \cite{HechtLanceTrelat2026}, and discretized neural ODEs \cite{chen2018neural}. A detailed investigation of these applications is beyond the scope of the present paper. We also point out that the natural connections between conservativity, differentiable programming, and algorithmic differentiation make our approach amenable to modern computational frameworks \cite{bolte2020mathematical,bolte2021conservative}. In Section~\myref{sec:diff-prog-and-appli}, we describe these connections with algorithmic differentiation as well as comparisons with alternatives to the adjoint formula in this context.

We discuss the computational cost of our adjoint approach from the perspective of differentiable programming, comparing it with other standard methods that are generally more costly; see Section~\myref{sec:diff-prog-and-appli}. 
To illustrate further this connection with applications, we provide a basic gradient-like algorithm based on the adjoint method in Section~\myref{sect:implicopt}. Our results show that the gradient method based on \eqref{form:asm} as a first order oracle, exhibits a minimizing behavior: generically, it is attracted to critical points of the value function $f$ for vanishing step sizes; see Proposition~\myref{cor:smallStepMethod}. This result can be extended to more advanced first order optimization algorithms; see Section~\myref{sect:implicopt}. 
% %%%%%%%%%%%%%%%%%%%%%%%%%%%%%%%%%%%%%%%%%%%%%%%%%%%%%%%%%%%%%%%%

\runinsubsectionstar[\normalsize]{Related work on the adjoint method} The term ``adjoint'' comes from the adjoint-state in the literature on control and PDE-constrained optimization, in particular shape optimization. The method is closely related to the field of sensitivity analysis in constrained programming and to the notion of parametric optimality in nonsmooth analysis. Let us discuss references and connections to our work.

A general presentation of sensitivity analysis for shape optimization can be found in \cite{zolesio2011shapes}. The adjoint formula in \eqref{form:asm} was presented, to our knowledge, in the context of shape optimization \cite{cea1986conception}. In this context, the adjoint formula naturally occurs in infinite dimension and its main objects of interest are directional derivatives. Subsequent developments \cite{delfour1988shape} were based on the connection with min-max optimization \cite{correa1985directional}.
Recent extensions can be found in \cite{delfour2017parametric, delfour2019control, delfour2025semidifferential}. These developments share a lot of similarities with results developed for sensitivity analysis of constrained programs.

Sensitivity analysis is a broad field of research, which goes beyond shape analysis, a general account is given in \cite{bonnans2000perturbation} in infinite dimension. We focus the discussion on connections with the adjoint formula in \eqref{form:asm} applied to problem \eqref{eq:fvdef}, which initially occurred in a nonlinear programming context. Early results date back to Gauvin \cite{gauvin1977differential} with many subsequent extensions \cite{fiacco1979extensions,Auslender1979,Janin1984} see the review in \cite{pacaud2025sensitivity} . A naive application of these results to our problem \eqref{form:asm} would lead, at best, to only directional derivatives, rather than a directly computable gradient-like object for the value function, as provided by our approach. Moreover, it would require, when possible, solving costly optimization problems over the sets of primal and dual solutions. Directly obtaining a gradient using this sensitivity analysis approach requires strong uniqueness conditions as in the original argument of Céa. This is visible for example in \cite{HechtLanceTrelat2026} where uniqueness is assumed, in the context of PDE constrained optimization, and \cite{demiguel2008decomposition} where second order sufficiency conditions are assumed in the context of decomposition methods.

Beyond directional derivatives, let us mention the parametric optimality approach. \cite{hiriart1978gradients} and \cite[Theorem 10.13]{rockafellar1998variational} established that the Clarke subdifferential of the value function is included in a convex set constructed with the subdifferential of the Lagrangian; see the discussion in Section~\myref{sect:parametricOptimality}. However, as noted in the simple Example~\myref{ex:failclarke}, this inclusion is strict in general, showing that subdifferential theory is not adapted to this subject. A limitation that this paper overcomes through conservative gradients.
% %%%%%%%%%%%%%%%%%%%%%%%%%%%%%%%%%%%%%%%%%%%%%%%%%%%%%%%%%%%%%%%%
\runinsubsectionstar[\normalsize]{Notations} Let \(\ell, L \in \N\) be natural numbers. We denote by \(\scalarp{\cdot, \cdot}\) the canonical Euclidean scalar product on \(\Rell\) and \(\norm{\cdot}\) its associated norm. For \(r>0\) and \(\vx \in \Rell\), we denote by \(\Bcal_c(\vx,r)\) the closed ball centered at \(\vx\) with radius \(r\), and by \(\Bcal(\vx,r)\) its open counterpart. Let \(A \subset \Rell\). We write \(\cl A\), \(\interior A\), and \(\conv A\) for the closure, interior, and convex hull of \(A\), respectively. The Jacobian and partial Jacobian of a vector-valued function \(\varphi\colon \Rell \to \R^{L}\) at \(\vz = (\vx, \vy) \in \Rell\) are denoted by \(\jac\varphi(\vz)\) and \(\jac_{x}\!\varphi(\vx, \vy)\), respectively. In the same way, the gradient and partial gradient of a real-valued function \(\varphi\colon \Rell \to \R\) at \(\vz = (\vx, \vy) \in \Rell\) are written as \(\nabla \varphi(\vz)\) and \(\nabla_{\!x}\varphi(\vx, \vy)\), respectively. For \(s,t\in \R\), we will sometimes use the notation \( \varphi(t^+)\) for \(\displaystyle \lim_{s\downarrow t} \varphi(s) \coloneq \lim_{s>t, s \to t} \varphi(s)\) and the notation \( \varphi(t^-) \) for \( \displaystyle \lim_{s\uparrow t} \varphi(s) \coloneq \lim_{s<t, s \to t} \varphi(s)\).
%%%%%%%%%%%%%%%%%%%%%%%%%%%%%%%%%%%%%%%%%%%%%%%%%
%%%%%%%%%%%%%%%%%%%%%%%%%%%%%%%%%%%%%%%%%%%%%%%%
%%%%%%%%%%%%%%%%%%%%%%%%%%%%%%%%%%%%%%%%%%%%%%%%%%%%%%%%%%%%%%%%
% %%%%%%%%%%%%%%%%%%%%%%%%%%%%%%%%%%%%%%%%%%%%%%%%%%%%%%%%%%%%%%%%
\section{The adjoint state method: smooth case}\label{sect:diffcase}

This section studies the adjoint state method in the setting where the objective function $F$ is continuously differentiable, in the presence of equality and inequality constraints and possibly a continuum of minimizers. We focus here on the smooth case for simplicity and will expand on nonsmooth objectives in Section~\myref{sect:generalresult}. 
%%%%%%%%%%%%%%%%%%%%%%%%%%%%%%%%%%%%%%%%%%%%%%%%%
%%%%%%%%%%%%%%%%%%%%%%%%%%%%%%%%%%%%%%%%%%%%%%%%
\subsection{Preliminaries} 
We consider the following parametric nonlinear programming problem:
\begin{align}\label{prob:cpo}
\begin{aligned}
\minimize_{\pvar \in \Rn}\quad & F(\pvar, \param) \\
\st\quad      & G(\pvar,\param)\le 0,\\
              & H(\pvar,\param)=0,
\end{aligned}
\end{align}
where \(\param\in\Rq\). The objective function \(F:\Rn\times\Rq\to\R\) and the constraint mappings
\[
G(\pvar,\param)=\transp{\big(g_1(\pvar,\param),\ldots,g_m(\pvar,\param)\big)}, 
\qquad 
H(\pvar,\param)=\transp{\big(h_1(\pvar,\param),\ldots,h_p(\pvar,\param)\big)}
\]
are all continuously differentiable, the inequality being understood coordinate-wise. For any \((\pvar, \param) \in \Rn \times \Rq \), let \(I(\pvar, \param) \subset \{1, \ldots, m\}\) be the set of active (inequality) constraints, i.e., for all \(i \in I(\pvar, \param)\), \(g_i(\pvar, \param) = 0\). The \emph{Lagrangian} of Problem \eqref{prob:cpo}, \(\Lcal\colon  \Rn \times \Rm \times \Rp \times \Rq \to \R \), writes
\begin{equation*}
    \Lcal(\pvar, \dvar, \dvarr, \param) \coloneq F(\pvar, \param) + \transp{\dvar} G(\pvar, \param) + \transp{\dvarr} H(\pvar, \param).
\end{equation*}
 
\begin{definition}[\emph{Mangasarian-Fromovitz constraint qualification}]\label{def:mfcq}
   Let \((\pvar, \param) \in \Rn \times \Rq\) such that \(G(\pvar, \param) \leq 0,  H(\pvar, \param) = 0\). The \emph{MFCQ} at \((\pvar, \param)\) states that the following implication holds
    \begin{align}\label{eq:22122025b}
     &\sum_{i \in I(\pvar, \param)} \dvar_i \nabla_{\!\pvar} g_i(\pvar, \param) + \transp{\jac_{\pvar}\!H(\pvar, \param)}\mu = 0, \text{ and } \dvar_i \geq 0, \forall\, i \in  I(\pvar, \param) \nonumber \\
     &\implies \mu = 0, \dvar_i = 0, \forall\, i \in  I(\pvar, \param).
    \end{align}
\end{definition}

\begin{remark}
    \label{rem:jointMFCQ}
    If the MFCQ is satisfied at \((\pvar, \param)\), then the following holds for the joint Jacobians in \((\pvar,\param)\), i.e,
    \begin{align*}
       &\transp{\jac\!G(\pvar, \param)}\dvar + \transp{\jac\!H(\pvar, \param)}\dvarr = 0, \text{ and } \dvar_i \geq 0, \forall\, i \in  I(\pvar, \param) \nonumber \\
     &\implies \dvarr = 0, \dvar_i = 0, \forall\, i \in  I(\pvar, \param).
    \end{align*} 
    In other words, the MFCQ holds taking into account the joint dependence in \(\pvar,\param\), not just the partial dependence in \(\pvar\)
\end{remark}

\begin{assumptions}\label{ass:specialcase}
The following assumptions hold throughout the current section, Section~\myref{sect:diffcase}.
\begin{enumerate}[label=\rm{(\roman*)}, ref=\rm{\roman*}, align=left, labelwidth=*, leftmargin=*, labelindent=0.5em, labelsep=-0.4em]

    \item The objective function \(F\) is continuously differentiable and semialgebraic, i.e., its graph is described by polynomial inequalities and equalities. The constraint functions \(G\) and \(H\) are also continuously differentiable and semialgebraic. 
    
    \item\label{ass:as2} The MFCQ is satisfied at any \((\pvar, \param) \in \Rn \times \Rq\) such that \(G(\pvar, \param) \leq 0,  H(\pvar, \param) = 0\).
    
    \item\label{ass:as3} For all \(\param \in \Rq\), the solution set 
    \[\{\pvar \in \Rn: F(\pvar, \param) = f(\param), G(\pvar,\param) \leq 0, H(\pvar, \param) = 0\}
    \] 
    is nonempty  and bounded by \(M_{\param} \geq 0\).
    In addition, for all \(\param \in \Rq\), there exist a neighborhood \(\Ncal\) of \(\param\) and \(M \geq 0\) such that for any \(\bar{\param} \in \Ncal\), \(M_{\bar{\param}} \leq M\).
\end{enumerate}
\end{assumptions}
\noindent We consider a selection mapping, \(\solvp \colon \Rq \to \Rn\), such that for all \(\param \in \Rq\), 
\begin{equation}\label{eq:solverp}
\solvp(\param) \in \{\pvar \in \Rn: F(\pvar, \param) = f(\param), G(\pvar,\param) \leq 0, H(\pvar, \param) = 0\}, 
\end{equation}
Thanks to Assumption~\myhyperref{ass:specialcase}{ass:as3}, for any \(\param \in \Rq\), the set in \eqref{eq:solverp} is nonempty, so \(\solvp\) is well defined. We use \(\solvp\) to represent a solver that returns, for a given \(\param \in \Rq\), a primal solution to the problem \eqref{prob:cpo}.

Let us define similarly a primal-dual solver, \(\solvpd \colon \Rq \to \Rn \times \Rm \times \Rp\), such that for all \(\param \in \Rq\),
\begin{align}\label{eq:solverpd}
    \solvpd(\param) &\in \big\{(\pvar,\dvar,\dvarr)\in \Rn \times \Rm \times \Rp: F(\pvar, \param) = f(\param), G(\pvar,\param) \leq 0, H(\pvar, \param) = 0,\nonumber \\
    &\qquad \dvar \geq 0,  \transp{G(\pvar,\param)}\dvar = 0, \text{ and } \nonumber \\
    &\qquad \nabla_{\!\pvar}F(\pvar, \param) + \jac_{\pvar}\!\transp{G(\pvar, \param)}\dvar + \jac_{\pvar}\!\transp{H(\pvar, \param)}\dvarr = 0\big\}.
\end{align}

In the definition of \(\solvpd\), the first constraint indicates that $x$ achieves the minimal value in \eqref{eq:fvdef} and the rest form the KKT system: feasibility, complementarity and stationarity. 
By Assumptions~\myeqhyperrefrange{ass:specialcase}{ass:as2}{ass:as2}{ass:as3}, for any \(\param \in \Rq\), the set in \eqref{eq:solverpd} is nonempty: the MFCQ is sufficient to ensure that KKT conditions are necessary for optimality (see, for example, \cite[Corolary 6.15]{rockafellar1998variational}), and \(\solvpd\) is well defined. For each \(\param \in \Rq\), the mapping \(\solvpd\)  outputs a primal solution $x$ to Problem \eqref{prob:cpo} and one of its associated Lagrange multipliers, a dual solution $(\dvar,\dvarr)$.
%%%%%%%%%%%%%%%%%%%%%%%%%%%%%%%%%%%%%%%%%%%%%%%%%%%%%%%%%%%%%%%%
%%%%%%%%%%%%%%%%%%%%%%%%%%%%%%%%%%%%%%%%%%%%%%%%%%%%%%%%%%%%%%%%
\subsection{The adjoint method}
We propose the use of the adjoint formula from \eqref{form:asm}, without imposing any uniqueness or smoothness requirement neither on the solutions to \eqref{prob:cpo}, nor on their associated dual/adjoint variables (Lagrange multipliers).
\renewcommand{\thealgorithm}{ASM}
\begin{algorithm}[H]
\captionsetup{name=Method}
\caption{Adjoint State Method}\label{alg:asm}
\begin{algorithmic}[1]
\Require \(\bar{\param} \in \Rq\)
\State  \((\pvar^*, \dvar^*, \dvarr^*) \gets \solvpd(\bar{\param})\). %
\State Compute:
\begin{align}\label{eq:adjcea}
    \vu \gets \nabla_{\!\param} \Lcal(\pvar^*, \dvar^*, \dvarr^*, \bar{\param}) = \nabla_{\!\param}F(\pvar^*, \bar{\param}) + \jac_{\param}\!\transp{G(\pvar^*, \bar{\param})}\dvar^* + \jac_{\param}\!\transp{H(\pvar^*, \bar{\param})}\dvarr^*.
\end{align}
\Ensure \(\vu\) \qquad \mbox{(surrogate gradient for the value function)}
\end{algorithmic}
\end{algorithm}
Let us discuss the connection with the adjoint formula as originally formulated by C\'ea \cite{cea1986conception}, under equality constraints.
We have, for \((\pvar^*, \dvar^*, \dvarr^*) = \solvpd(\bar{\param})\), that
\begin{align}\label{eq:adjceatwo}
    \jac_{\pvar}\!\transp{G(\pvar^*, \bar{\param})}\dvar^* + \jac_{\pvar}\!\transp{H(\pvar^*, \bar{\param})}\dvarr^* = - \nabla_{\!\pvar}F(\pvar^*, \bar{\param}).
\end{align}
In \eqref{eq:adjceatwo} and \eqref{eq:adjcea}, one can recognize the operations of the adjoint state method proposed by \cite[Equations (2.7) and (2.9)]{cea1986conception} with \(m = 0\) (no inequality constraint). The theoretical validity of the formula in \cite{cea1986conception}, is based on the uniqueness and differentiability of both the primal and dual solutions, essentially requiring strong second order sufficiency conditions. The use of conservative fields enables us to consider the exact same formula as \cite{cea1986conception}, without these strong conditions, allowing for possible nonsmooth or non-unique primal-dual solutions.

The following is our main result regarding the adjoint state method~\emph{\myref{alg:asm}} for smooth objectives. It is a particular case of Theorem~\myref{theo:main} in Section~\myref{sect:generalresult}, where the result is proved. 

\begin{theorem}\label{theo:mainsmooth}
    Let Assumptions~\myref{ass:specialcase} stand. Then, the adjoint state method~\emph{\myref{alg:asm}} is a selection of a \hyperref[def:conservative]{conservative~field\,} for \(f\), i.e.,
    \[
    \emph{ASM}(\param) \in D_f(\param) \quad \forall\,\param \in \Rq,
    \]
    where \(D_f\) is a conservative~field for \(f\). In particular, for any locally Lipschitz $\param \colon \R \to \Rq$, we have
    \begin{align*}
        \frac{d}{dt} f(\param(t)) = \left\langle \emph{ASM}(\param(t)), \dot{\param}(t) \right\rangle \text{ for almost all } t \text{ in } \R.
    \end{align*}
\end{theorem}
The conservative field \(D_f\) in Theorem~\myref{theo:mainsmooth} is rigourously defined in Theorem~\myref{theo:main}. Formally, it consists of all the possible outputs of the method~\myref{alg:asm} for all possible choice of \(\solvpd\) as in the right hand side of \eqref{eq:solverpd}.

The same result applies \textit{mutatis-mutandis} to more general constraints of the form \(C(\pvar,\param) \in K\) for a closed convex cone \(K\). In this section, we focus on the simpler and more widespread form \eqref{eq:fvdef} for simplicity; see Section~\myref{sect:generalresult} for a general statement. This allows to cover conic programs such as paramteric second order cone programs (SOCP) and semidefinite programs (SDP).

We highlight the following important properties of conservative fields.

\begin{remark}[On conservative fields]\label{rmk:consequenceasm} The result in Theorem~\myref{theo:mainsmooth} allows to invoke generic properties of conservative fields. Among them, let us highlight the following.
\begin{enumerate}[label=\rm{(\roman*)}, ref=\rm{\roman*}, align=left, labelwidth=*, leftmargin=*, labelindent=0.5em, labelsep=-0.4em]

    \item Conservative fields represent a notion of generalized gradients \cite{bolte2021conservative}. Theorem~\myref{theo:mainsmooth} provides a rigorous mathematical framework and analysis that justify the use of the adjoint state method without any uniqueness or differentiability assumptions on the primal/state and dual/adjoint solutions. The standard notions of subdifferential from nonsmooth analysis \cite{rockafellar1998variational} are insufficient to describe the variational properties of the adjoint method under those conditions; see Example~\myref{ex:failclarke}. This justify the relevance of the notion of conservative field in this context.

    \item For general locally Lipschitz semialgebraic functions, like \(f\), selections of conservative fields can be used for optimization. They carry sufficient variational information to serve as oracle for first-order methods with qualitative guaranties; see \cite{bolte2021conservative, xiao2024adam}. This fact enables, for instance, parameter optimization for \eqref{prob:cpo} through first-order methods by employing the adjoint state formula to efficiently compute a gradient-like oracle for \(f\) at each iteration; see Section~\myref{sect:implicopt}.
    
    \item Conservative fields are compatible with differentiable programming paradigms such as algorithmic differentiation \cite{bolte2021conservative, bolte2020mathematical}. This means that, in practice, \myref{alg:asm} can be combined with further operations in a larger numerical program, in a way which is compatible with algorithmic differentiation as implemented for example in Python libraries such as Pytorch; see Section~\myref{sec:diff-prog-and-appli}.

    \item The combination of \cite[Proposition 2.1.2]{clarke1983nonsmooth} and \cite[Corollary 1]{bolte2021conservative} allows to obtain the following bound on Dini directional derivatives for any \(\param,d \in \Rq\):
    \begin{align*}
        &\min_{v \in \consv_f(\param)} \left\langle v, d\right\rangle \leq \liminf_{t\downarrow0}\frac{f(\param + td) - f(\param)}{t} \leq \limsup_{t\downarrow0}\frac{f(\param + td) - f(\param)}{t} \leq \max_{v \in \consv_f(\param)} \left\langle v, d\right\rangle.
    \end{align*}
    This is reminiscent of the sensitivity analysis performed, for example, in \cite{gauvin1977differential,bonnans2000perturbation,pacaud2025sensitivity}, where more precise results on directional derivatives are proposed.
\end{enumerate}
\end{remark}

\begin{remark}[On the semialgebraic assumption]
More generally, Theorem~\myref{theo:mainsmooth} and all its consequences in Remark~\myref{rmk:consequenceasm} remain valid if all the functions, \(F, G \text{ and } H\), are definable in the same o-minimal structure; see Definition~\myref{def:definable}. This allows to include transcendental operations in the definition of \(F, G \text{ and } H\), such as exponential or logarithm, and covers virtually all applications which can be considered in practice.
\end{remark}

\begin{example}\label{ex:failclarke} The following example shows that the (Clarke) subdifferential cannot capture all the information of the output of the adjoint formula in \myref{alg:asm} as the latter may produce artifacts which do not carry any standard variational meaning.

Set $n=2$ and $m=3$ with
\begin{align*}
    F(\pvar,\param) &= \pvar_1 \pvar_2,
    &G(\pvar,\param) = ( \pvar_1,
        \pvar_2 - \param,
        \pvar_2 + \param
    )\!\transp\;.
\end{align*}
The constraints are expressed equivalently $\pvar_1 \leq 0$ and $\pvar_2 \leq - |\param|$. We have $f(\param) = 0$ for all $\param \in \R$, and the MFCQ holds. For $\bar{\param} = 0$, we have the following KKT solution
\begin{align*}
    \pvar^* &= 
    \begin{pmatrix}
        -1 \\ 0
    \end{pmatrix}, &
    \dvar^* 
    = \begin{pmatrix}
        0\\ 0\\ 1
    \end{pmatrix},
\end{align*}
such that $\nabla_{\!\param} \Lcal(\pvar^*, \dvar^*, \bar{\param}) = 1$, which is not a subgradient of $f$.
\end{example}
%%%%%%%%%%%%%%%%%%%%%%%%%%%%%%%%%%%%%%%%%%%%%%%%%%%%%%%%%%%%%%%%%%%%%%%%%%%%%%%%%
%%%%%%%%%%%%%%%%%%%%%%%%%%%%%%%%%%%%%%%%%%%%%%%%%%%%%%%%%%%%%%%%%%%%%%%%%%%%%%%%%

\subsection{Algorithmic differentiation and alternative methods}\label{sec:diff-prog-and-appli}

Algorithmic differentiation implements the idea of differentiating numerical programs in an efficient and user friendly way \cite{griewank2008evaluating}. This is one of the main building blocks of neural network training in machine learning. In its most simple form, a numerical program is simply a composition $P = f_1 \circ f_2 \circ \ldots \circ f_L$, where each $f_i$ is taken from a dictionary of basic numerical functions, compatible with the differentiable programming framework. The user specifies the structure of the composition with a numerical program, and algorithmic differentiation combines automatically derivative objects $d_1,d_2\ldots,d_L$ associated to each $f_i$, to evaluate the derivative of $P$ numerically. Theorem~\myref{theo:mainsmooth} states that for any $\theta$, $\text{ASM}(\param) \in D_f(\param)$, where $f$ is the value function in \eqref{eq:fvdef}, and $D_f$ is a conservative field for $f$. Conservative fields are compatible with algorithmic differentiation \cite{bolte2021conservative}. Therefore, an interpretation of Theorem~\myref{theo:mainsmooth} is that $\text{ASM}$ can be used as a gradient surrogate $d$ for the value function $f$ in a differentiable programming framework. Let us discuss potential alternatives from the algorithmic differentiation literature \cite{griewank2008evaluating}.

For the remainder of this section \(\bar{\param} \in \Rq\) denotes a fixed set of parameters. We compare \myref{alg:asm} with the following two widely used approaches that are based on differentiating the solution to get the sensitivities of \(f\). The first approach, which we will simply call \emph{automatic differentiation}, uses reverse-mode algorithmic differentiation to obtain the derivative of the solution \(\pvar(\bar{\param})\) by differentiating through \(\solvp\). The second is \emph{implicit differentiation}, where the implicit function theorem is applied to the critical condition \(\nabla \Lcal_{1}(\solvpd(\bar{\param})) = 0\), with  \(\Lcal_{1} \colon (\pvar, \dvar, \dvarr) \mapsto \Lcal(\pvar, \dvar, \dvarr, \bar{\param})\) where \((\pvar, \dvar, \dvarr) \in \Rn \times \Rm \times \Rp\) and \(\bar{\param}\) is fixed.

In terms of assumptions, it is immediately apparent that both methods require the uniqueness and differentiability of the solutions, while \myref{alg:asm} does not. For the purpose of comparing computational costs, let us consider the favorable settings ensuring the validity of both methods. 

A straightforward way to implement \myref{alg:asm}, after getting a primal-dual solution from \(\solvpd\), is for the user to just compute \(u = \texttt{autodiff}\Lcal_{2}(\bar{\param})\), with \(\Lcal_{2} \colon \param \mapsto \Lcal(\solvpd(\bar{\param}), \param)\), where \(\bar{\param}\) is fixed, treated as a constant in the differentiation process. We use the notation \(\text{Cost}\) to denote the computational cost of implementing operations, which can be identified with computational time. According to the cheap gradient principle \cite{griewank2008evaluating},  
    \[
    \text{Cost}\big(\texttt{autodiff}\Lcal_{2}(\bar{\param})\big) \leq \omega \text{Cost}\big({\Lcal_{2}}\big) = \omega\text{Cost} \big( \Lcal \big),
    \]
    where \(\omega > 0\) is an absolute constant, \(\text{Cost}\big(\texttt{autodiff}\Lcal_{2}(\bar{\param})\big)\) is the computational cost to calculate the derivative of \(\Lcal_{2}\) through reverse-mode automatic differentiation, and \(\text{Cost}\big(\Lcal_{2}\big)\) and \(\text{Cost}\big(\Lcal\big)\) the computational cost to evaluate the value of the functions \(\Lcal_{2}\) and \(\Lcal\), respectively. Therefore, the total cost of the adjoint state method~\myref{alg:asm} is upper bounded as follows:
    \begin{equation}\label{eq:costasm}
       \text{Cost}\big(\text{ASM}\big) \leq \text{Cost}\big(\solvpd\big) + \omega\text{Cost}\big(\Lcal\big).
    \end{equation}
    Typically, the term \(\text{Cost}\big(\solvpd\big)\) dominates on the right hand side.
\begin{remark}\label{rmk:costdualvar}
    In most nonlinear programming solvers, such as SNOPT \cite{gill2005snopt} and IPOPT \cite{wachter2006implementation}, dual variables are already computed internally as part of satisfying the KKT optimality conditions, or because of the primal-dual nature of the underlying algorithms. This also applies to conic programming solvers such as SDPA \cite{yamanishita2003SDPA}, SDPT3 \cite{toh1999sdpt3} or commercial solvers such as MOSEK or KNITRO. Consequently, in such a situation, getting the dual solution in addition to the primal solution, requires little to no additional computational effort, \(\text{Cost}\big(\solvpd\big)\) is the same as \(\text{Cost}\big(\solvp\big)\) and the adjoint state formula \myref{alg:asm} is readily implementable with minimal code extension.
\end{remark}

Both vanilla automatic differentiation and implicit differentiation are respectively based on the following chain rules,
\begin{align}
    \nabla f(\bar{\param}) &= \transp{\nabla_{\!\pvar} F\big(\solvp(\bar{\param}), \bar{\param}\big)}\jac\big(\solvp(\bar{\param})\big) + \nabla_{\!\param} F\big(\solvp(\bar{\param}), \bar{\param}\big) \nonumber \\
    &= \transp{\nabla \Lcal_{1}\big(\solvpd(\bar{\param})\big)}\jac\big(\solvpd(\bar{\param})\big) + \nabla \Lcal_{2}(\bar{\param}) \nonumber,
\end{align}
The computational cost of automatic differentiation can be roughly bounded as follows:
\begin{align}\label{eq:costnaivead}
    \text{Cost}\big(\text{AD}\big) &=\text{Cost}\big(\solvp\big) + \text{Cost}\big(\texttt{autodiff}F(\solvp(\bar{\param}), \bar{\param})\big) \nonumber \\
    &\qquad + \text{Cost}\big(\texttt{autodiff}\,\solvp(\bar{\param})\big) \nonumber \\
    &\leq \text{Cost}\big(\solvp\big) + \omega\text{Cost}\big(F\big) + \omega \text{Cost}\big(\solvp\big).
\end{align}
Again, in typical situations the terms \(\text{Cost}\big(\solvp\big)\) dominates.
Finally, assuming all regularity conditions for its application \cite{amos2017optnet}, implicit differentiation has the following approximate bound on its computational cost:
\begin{align}\label{eq:costid}
    \text{Cost}\big(\text{ID}\big) &= \text{Cost}\big(\solvpd\big) + \text{Cost}\big(\texttt{autodiff}\Lcal(\solvpd(\bar{\param}), \bar{\param})\big)\nonumber \\ 
    &\qquad + \text{Cost}\Big(\texttt{autodiff}\big(\texttt{autodiff}\Lcal\big)(\solvpd(\bar{\param}),\bar{\param})\Big) + \text{Cost} \big(\text{LE}\big) \nonumber \\
    &\leq \text{Cost}\big(\solvpd\big)  + \text{Cost} \big(\text{LE}\big) + \omega(n+q+1)\text{Cost}\big(\Lcal\big),
\end{align}
where \text{Cost}\big(\text{LE}\big) is the computational cost of inverting a linear system.

The choice of automatic differentiation and implicit differentiations in \eqref{eq:costnaivead} and \eqref{eq:costid} depends on the number of variables and number of constraints, and possibly further characteristics of the problem. Depending on the setting, one solution may be preferable to the other, both in terms of worst case upper bounds and in practice. Note that this is a high level analysis and the actual performance may depend on additional computational factors, such as how second order derivatives are computed, how linear equations are solved, and potential approximation and parallelization strategies. Nonetheless this provides a reasonable upper bound on the computational burden for both automatic differentiation and implicit differentiation. Let us also mention that automatic differentiation requires to have access to the code of \(\solvp\), which may not be available for closed-source software.

Comparing with Equations~\eqref{eq:costasm}, it is obvious that the worst case computational cost of \myref{alg:asm} is lower, and possibly much lower, than either automatic differentiation or implicit differentiation, making it a strong competitor for the purpose of differentiating the value function of nonlinear programs. 
%%%%%%%%%%%%%%%%%%%%%%%%%%%%%%%%%%%%%%%%%%%%%%%%%%
%%%%%%%%%%%%%%%%%%%%%%%%%%%%%%%%%%%%%%%%%%%%%%%%
%%%%%%%%%%%%%%%%%%%%%%%%%%%%%%%%%%%%%%%%%%%%%%%%%%%%%%%%%%%%%%%%
% %%%%%%%%%%%%%%%%%%%%%%%%%%%%%%%%%%%%%%%%%%%%%%%%%%%%%%%%%%%%%%%%
\section{Elements of nonsmooth analysis and o-minimality}\label{sect:prelim}

In full generality, our study is conducted in a setting where the objective function \(F\) is merely locally Lipschitz. Accordingly, we begin by introducing tools required to handle nonsmoothness.

\subsection{Conservative set-valued fields}

A \emph{set-valued map} or \emph{multifunction} \(D\colon \Rell \rightrightarrows \Rs\) is a mapping from \(\Rell\) to the set of all subsets of \(\Rs\). The \emph{graph} of \(D\) is given by \(\graph D:=\{(\vx, \vy) \in \Rell \times \Rs : \vy \in D(\vx)\}\). \(D\) is \emph{locally bounded} at \(\vx \in \Rell\) if there exist a neighborhood \(\Ncal\) of \(\vx\) and \(r > 0\) such that \(\displaystyle \bigcup_{\vz \in \Ncal} D(\vz) \subset \Bcal_c(0, r)\). \(D\) is \emph{graph closed} if \(\graph D\) is a closed subset of \(\Rell \times \Rs\). Equivalently, \(D\) is graph closed if for all \((\vx_k)_{k\in \N} \subset \Rell\) and all \((\vy_k)_{k\in \N} \subset \Rs\) such that \(\vx_k \xrightarrow[k \to \infny]{} \vx\), \(\vy_k \xrightarrow[k \to \infny]{} \vy\) and \(\vy_k \in D(\vx_k)\) for any \(k \in \N\), it follows that \(\vy \in D(\vx)\).

\begin{definition}[Clarke generalized gradients \cite{clarke1983nonsmooth}]\label{def:clarkesubdiff} Let \(\varphi\colon \Rell \to \R\) be a locally Lipschitz function. By Rademacher's theorem, \(\varphi\) is differentiable on a full measure subset of \(\Rell\), say \(\Omega\). Then, the \emph{Clarke subdifferential} of \(\varphi\) is the set-valued map \(\clarke \varphi\colon \Rell \rightrightarrows \Rell\) defined as
\begin{align*}
    \vx \mapsto \conv\,\{u \in \Rell: \exists\, (u_k)_{k\in \N} \subset \Omega, u_k \xrightarrow[k \to \infny]{} \vx \text{ and }  \nabla \varphi(u_k) \xrightarrow[k \to \infny]{} u \}.
\end{align*}
 The Clarke subdifferential of \(\varphi\), like for any other locally Lipschitz real-valued function, has nonempty, convex and compact values, and is closed graph and locally bounded. 
 
 Given \((\vy,\vz) \in \Rell\), the Clarke partial subdifferential \(\clarkep{\vy}\varphi(\vy,\vz)\) of \(\varphi\) at \((\vy,\vz)\) with respect to \(\vy\) is the Clarke subdifferential of the function \(\varphi(\cdot,\vz)\) at \(\vy\).
\end{definition}

Our analysis revolves around the notion of conservative gradients that we present in the next definition. Let \(I \subset \R\) be a nonempty interval and \(\vu\colon I \to \Rell\) be a \emph{curve}. The curve \(\vu\) is said to be \emph{absolutely continuous} if \(\vu\) is differentiable Lebesgue-almost everywhere on \(I\),  its derivative \(\dot{\vu}\) is Lebesgue integrable on \(I\), and \(\displaystyle \vu(b) - \vu(a) = \int_{a}^{b} \dot{\vu}(t) dt\), for all \(a,b \in I\).

\begin{definition}[Conservative fields \cite{bolte2021conservative}]\label{def:conservative} Let \(\varphi\colon \Rell \to \R\) be a locally Lipschitz function and \(\consv_{\varphi}\colon \Rell \to \Rell\) be a locally bounded and graph closed set-valued map with nonempty values. We say that \(\consv_{\varphi}\) is a conservative field for \(\varphi\) if for any absolutely continuous curve \(\vu\colon[0, 1] \to \Rell\),
    \begin{equation}\label{eq:chainrule}
        \frac{d}{dt}\varphi(\vu(t)) = \scalarp{\dot{\vu}(t), \vv}  \quad \forall\,\vv \in \consv_{\varphi}(\vu(t)),
    \end{equation}
for almost all \(t \in [0, 1]\). Any such \(\varphi\) is called \emph{path differentiable}.
\end{definition}
%%%%%%%%%%%%%%%%%%%%%%%%%%%%%%%%%%%%%%%%%%%%%%%%%%%%%%%%%%%%%%%%
%%%%%%%%%%%%%%%%%%%%%%%%%%%%%%%%%%%%%%%%%%%%%%%%%%%%%%%%%%%%%%%%
\subsection{Definition of o-minimality } 
We briefly recall the concepts of tame geometry relevant to the class of mappings (\emph{definable mappings}) considered in the present paper. Important works on this topic include~\cite{van1996geometric, coste2000introduction}.
\begin{definition}[o-minimal structures] 
An \emph{o-minimal structure} on \( (\R, +, \cdot) \) is a collection of families 
\( \Ocal = (\Ocal_\ell)_{\ell \in \N} \), where for each 
\( \ell \in \N \), the set \( \Ocal_\ell \) is a family of subsets of 
\( \Rell \), satisfying the following properties:
\begin{enumerate}[label=\rm{(\roman*)}, ref=\rm{\roman*}, align=left, labelwidth=*, leftmargin=*, labelindent=0.5em, labelsep=-0.4em]

\item \( \Ocal_\ell \) is closed under complementation, finite unions, and finite intersections.
\item If \( A \in \Ocal_\ell \), then both 
\( A \times \R \) and \( \R \times A \) belong to 
\( \Ocal_{\ell+1} \).

\item Let \( \pi : \R^{\ell+1} \to \Rell \) be the canonical projection.  
If \( A \in \Ocal_{\ell+1} \), then \( \pi(A) \in \Ocal_\ell \).

\item The family \( \Ocal_\ell \) contains all real algebraic subsets of 
\( \Rell \), which are sets of the form
\begin{equation*}
\{ x \in \Rell : \phi(x) = 0 \},
\end{equation*}
where \( \phi\colon \Rell \to \R\) is a polynomial function.

\item The sets in \( \Ocal_1 \) are precisely the finite unions of points and intervals.
\end{enumerate}
\end{definition}

\begin{definition}[Definable sets and definable mappings]\label{def:definable}
A subset of \( \mathbb{R}^\ell \) that belongs to an o-minimal structure 
\( \Ocal \) is said to be \emph{definable} in \( \Ocal\). 
A function or a set-valued map is definable in \( \Ocal \) if its graph is definable in \( \Ocal \).
\end{definition}

\begin{example}
    The simplest o-minimal structure is given by the class of \emph{real semialgebraic sets}. 
    Recall that a set \( A \subset \mathbb{R}^\ell \) is called \emph{semialgebraic} 
    if it is a finite union of sets of the form
    \begin{equation*}
    \bigcap_{k=1}^{\bar{k}}
    \{ x \in \mathbb{R}^\ell : \phi_k(x) < 0,\; \varphi_k(x) = 0 \},
    \end{equation*}
    where \( \bar{k} \geq 1 \) and \( \phi_k\colon \Rell \to \R\) and \(\varphi_k\colon \Rell \to \R\) are polynomial functions for all \(k \in \{1, \ldots, \bar{k}\}\). 
    Therefore, the semialgebraic functions and semialgebraic set-valued maps are functions and set-valued maps, respectively, whose graphs are semialgebraic sets.
\end{example}

From now on, we fix an o-minimal structure 
\( \Ocal \). And unless stated otherwise, all definable sets and definable mappings are understood to be definable in \( \Ocal \).
%%%%%%%%%%%%%%%%%%%%%%%%%%%%%%%%%%%%%%%%%%%%%%%%%
%%%%%%%%%%%%%%%%%%%%%%%%%%%%%%%%%%%%%%%%%%%%%%%%
%%%%%%%%%%%%%%%%%%%%%%%%%%%%%%%%%%%%%%%%%%%%%%%%%%%%%%%%%%%%%%%%
% %%%%%%%%%%%%%%%%%%%%%%%%%%%%%%%%%%%%%%%%%%%%%%%%%%%%%%%%%%%%%%%%
\section{The adjoint state method: general conservative case}\label{sect:generalresult}
%%%%%%%%%%%%%%%%%%%%%%%%%%%%%%%%%%%%%%%%%%%%%%%%%%%%%%%%%%%%%%%%
%%%%%%%%%%%%%%%%%%%%%%%%%%%%%%%%%%%%%%%%%%%%%%%%%%%%%%%%%%%%%%%%

In this section, we revisit the general problem \eqref{eq:closedcone}, which has the following form: for a closed convex cone \(K \subset \Rm\), a locally Lipschitz function \(F \colon \Rn \times \Rq \to \R\), and a continuously differentiable mapping \(C \colon \Rn \times \Rq \to \Rm\),
\begin{align}\label{prob:cpoCone}
    f(\param) \coloneq \inf_{\pvar \in \Rn} \left\{ F(\pvar,\param) \;:\; C(\pvar,\param) \in K \right\}.
\end{align}

Throughout the remainder of this section, all notation is adapted to Problem \eqref{prob:cpoCone}. In this setting, the MFCQ condition given in Equation \eqref{eq:22122025b} of Definition~\myref{def:mfcq} extends as follows. We denote by
\(
K^\circ \coloneq \left\{ x \in \Rm \;:\; \langle x,y\rangle \leq 0,\ \forall y \in K \right\}
\)
the polar cone of \(K\).

\begin{definition}[\emph{RCQ}]\label{def:mfcqCone}
   The Robinson's Constraint Qualification is satisfied at \((\pvar, \param) \in \Rn \times \Rq \) whenever
    \begin{align*}
     & \transp{\jac_{\pvar}\!C(\pvar, \param)}\dvar = 0, \,\dvar \in K^\circ , \, \left\langle C(\pvar, \param) , \dvar \right\rangle = 0 \implies \dvar = 0.
    \end{align*}
\end{definition}
The RCQ in Definition~\myref{def:mfcqCone} is one of the equivalent forms of Robinson's constraint qualification in finite dimension \cite[Corollary 2.98]{bonnans2000perturbation}.
Note that the condition on \(\dvar\) expresses the fact that $\dvar \in N_K(C(\pvar,\param))$, where $N_K$ denotes the normal cone to $K$ as follows.
\begin{lemma}\emph{(\cite[Corollary 23.5.4]{rockafellar1997convex})}
    Let $K \subset \Rm$ be a closed convex cone, then for any $z \in K$, we have $N_K(z) = \{\lambda \in K^\circ, \, \left\langle z, \lambda \right\rangle = 0\}$.
    \label{lem:normalConeCone}
\end{lemma}

The condition $\left\langle z, \lambda \right\rangle = 0$ in Lemma~\myref{lem:normalConeCone} is called complementarity. It is a specificity of convex cones and constitutes a crucial proof element.
The \emph{Lagrangian} of Problem \eqref{prob:cpoCone}, \(\Lcal\colon  \Rn \times \Rm  \times \Rq \to \R \), is the same as for \eqref{prob:cpo} and writes
\begin{equation*}
    \Lcal(\pvar, \dvar, \param) \coloneq F(\pvar, \param) + \transp{\dvar} C(\pvar, \param).
\end{equation*}
We recover \eqref{prob:cpo} using the following lemma. 
\begin{lemma}\label{lem:eqmfcqrcq}
    Set $K = \Rm_- \times \{0\} \subset \Rm \times \Rp$ and let $C \colon \Rn \times \Rq \to \Rm \times \Rp$ be continuously differentiable and defined blockwise with first component $G$ and second component $H$, then the MFCQ condition from Definition~\myref{def:mfcq} holds if and only if the RCQ condition in Definition~\myref{def:mfcqCone} holds.
\end{lemma}
\begin{proof}
    It is routine to check that \eqref{eq:22122025b} is the same as RCQ using the fact that $K^\circ = \Rm_+ \times \Rp$.
\end{proof}

\subsection{Preliminary results}
Define the set-valued map \(R: \Rq \rightrightarrows \Rn\) of \emph{feasible point} of Problem \eqref{prob:cpoCone} as
\begin{equation*}
    R(\param) \coloneq \left\{\pvar \in \Rn: C(\pvar, \param) \in K\right\},
\end{equation*}
and the set-valued map \(\solp\colon \Rq \rightrightarrows \Rn\) of \emph{primal solutions} of \eqref{prob:cpoCone} as
\begin{align*}
    \solp(\param) \coloneq \argmin_{} \left\{F(\pvar, \param): \pvar \in R(\param)\right\}.
\end{align*}
Note that, by continuity of $C$ and closedness of $K$, the set-valued map \(R\) is graph closed. 
Our main construction is based on the following.
\begin{assumptions}[Standing assumptions]\label{ass:standing} \hfill

\begin{enumerate}[label=\rm{(\roman*)}, ref=\rm{\roman*}, align=left, labelwidth=*, leftmargin=*, labelindent=0.5em, labelsep=-0.4em]

    \item The objective function \(F\) is locally Lipschitz and definable with definable conservative field \(D_F\). The constraint function \(C\) is continuously differentiable and definable, and \(K\) is a closed convex definable cone.
    
    \item\label{ass:a2} The RCQ is satisfied at any \((\param, \pvar) \in \graph R\), i.e., the RCQ from Definition~\myref{def:mfcqCone} holds throughout the feasible set \(R(\param)\) for any \(\param \in \Rq\).
    
    \item\label{ass:a3} The set-valued map \(\solp\) has nonempty values and is locally bounded. 
\end{enumerate}
\end{assumptions}

\begin{remark}\hfill
\begin{enumerate}[label=\rm{(\roman*)}, ref=\rm{\roman*}, align=left, labelwidth=*, leftmargin=*, labelindent=0.5em, labelsep=-0.4em]
    
    \item When $K$ is the negative orthant, the authors in \cite{bolte2018qualification} proved that, for diagonally perturbed definable programming, the Assumption~\myhyperref{ass:standing}{ass:a2} is generic, i.e., the MFCQ holds throughout the feasible set for all but finitely many (hence for Lebesgue-almost all) diagonal perturbations; Problem \eqref{prob:cpo} with no equality constraints corresponds to the zero diagonal perturbation. This result hints at the prevalence of the MFCQ in the definable context.
    
    \item Assumption~\myhyperref{ass:standing}{ass:a3} ensures that, for any \(\param \in \Rq\), the infimum in the definition of the value function \(f\), in \eqref{prob:cpoCone}, is achieved. Then \(f\) can be equivalently defined as \(\param \mapsto F(\solp(\param),\param)\).
\end{enumerate}    
\end{remark}
\noindent We set \(\consv_{F(\cdot,\param)} \colon \Rn \rightrightarrows \Rn  \), the projection on the \(\pvar\) component of \(D_F\), which provides a conservative fields for the partial function \(F(\cdot,\param)\), for fixed \(\param \in \Rq\). The KKT system for Problem \eqref{prob:cpoCone} is given by
\begin{align}
    \label{eq:KKTcone}
    0 \in \conv\consv_{F(\cdot,\param)}(\pvar) +  \transp{\jac_{\pvar}\!C(\pvar, \param)} \dvar, \quad C(\pvar,\param) \in K, \, \dvar \in K^\circ,\,\left\langle C(\pvar, \param) , \dvar \right\rangle = 0,
\end{align}
where the conditions on $\dvar$ express the fact that $\dvar \in N_K(C(\pvar,\param))$ as in Lemma~\myref{lem:normalConeCone}.
Let \(\kkt\colon \Rn \times \Rq \rightrightarrows \Rm \) be the set-valued map providing Lagrange multipliers satisfying stationarity and complementarity from the KKT system \eqref{eq:KKTcone}:
\begin{align*}
    \kkt(\pvar, \param) \coloneq \big\{\dvar \in \Rm :  0 \in \conv\consv_{F(\cdot,\param)}(\pvar) +  \transp{\jac_{\pvar}\!C(\pvar, \param)} \dvar, \, \, \dvar \in K^\circ, \transp{\dvar} C(\pvar, \param) = 0 \big\}.
\end{align*}
We then define the set-valued map of \emph{primal-dual solutions} \(\solpd\colon \Rq \rightrightarrows \Rn \times \Rm \) as 
\begin{align*}
    \solpd(\param) \coloneq \{ (\pvar, \dvar) \in \Rn \times \Rm  : \pvar \in \solp(\param), \dvar \in \kkt(\pvar, \param)\}.
\end{align*}

\begin{proposition}\label{prop:locboundp}
    The set-valued map of primal-dual solutions \(\solpd\) has nonempty values and is locally bounded.
\end{proposition}
\begin{proof} 
    For all \(\param \in \Rq\) and for every \(\pvar \in \solp(\param)\), the conservative KKT conditions~\eqref{eq:KKTcone} hold true by combining \cite[Theorem 3.3]{jourani1994constraint} and \cite[Corollary 1]{bolte2021conservative}, so that \(\kkt(\pvar, \param) \neq \varnothing\) and \(\solpd(\param)\) is nonempty.

    Let \(\param \in \Rq\). For the sake of contradiction, we assume that \(\solpd\) is not locally bounded around \(\param\). It follows that there exist \((\param_k)_{k \in \N} \subset \Rq\) and \((\pvar_k, \dvar_k) \in \solpd(\param_k)\), for all \(k \in \N\), such that \(\param_k \xrightarrow[k \to \infny]{} \param\) and \(\norm{(\pvar_k, \dvar_k)} \xrightarrow[k \to \infny]{} \infny\). Since \(\solp\) is locally bounded, we have that \((\pvar_k)_{k \in \N}\) is bounded and \(\norm{\dvar_k} \xrightarrow[k \to \infny]{} \infny\). 
    Passing to subsequences if necessary, we can assume that \(\norm{\dvar_k} \neq 0\) for all \(k \in \N\), and \(\pvar_k \xrightarrow[k \to \infny]{} \pvar\). By Lemma~\myref{lem:sclosedgraph}, we get \(\pvar \in \solp(\param)\). Let, for all \(k \in \N\), \(\alpha_k = \frac{\dvar_k}{\norm{\dvar_k}}\), so that \(\norm{\alpha_k} = 1\). Up to another subsequence, we can also assume \(\alpha_k \xrightarrow[k \to \infny]{} \alpha\) such that \(\norm{\alpha} = 1\). By definition of the set-valued map \(\kkt\), we know that, for every \(k \in \N\),
    \begin{equation}\label{eq:22122025a}
        \exists\, \vv_k \in \conv \consv_{F(\cdot,\param_k)}(\pvar_k), \quad\transp{\jac_{\pvar}\!C(\pvar_k, \param_k)}\alpha_k  = -\frac{\vv_k}{\norm{\dvar_k}}.
    \end{equation}
    Using the fact that \((\pvar, \param) \mapsto \consv_{F(\cdot,\param)}(\pvar)\) is locally bounded (because \(D_F\) is locally bounded), we have that \((\vv_k)_{k \in \N}\) is bounded. Up to another subsequence if needed, we have that \(\vv_k \xrightarrow[k \to \infny]{} \vv\). Taking the limit in {\eqref{eq:22122025a}}, we can write
    \begin{equation}\label{eq:21122025a}
        \transp{\jac_{\pvar}\!C(\pvar, \param)}\alpha = 0.
    \end{equation}
    Furthermore, for all \(k \in \N\), \(\alpha_{k} \in K^\circ\)  and \( \left\langle \alpha_{k}, C(\pvar_k, \param_k)\right\rangle = 0\).
    Passing to the limit, we get
    \begin{equation}\label{eq:21122025b}
        \alpha \in K^\circ\; \text{ and }\; \left\langle \alpha, C(\pvar, \param)\right\rangle = 0.
    \end{equation}
    Given Equations~\eqref{eq:21122025a} and \eqref{eq:21122025b}, the RCQ in Definition~\myref{def:mfcqCone} imposes \( \alpha = 0\). This contradicts the fact that \(\norm{\alpha} = 1\). Thus \(\solpd\) is locally bounded around \(\param\).
\end{proof}

\begin{remark} For any \(\param \in \Rq\), thanks to Proposition~\myref{prop:locboundp}, \(\solpd(\param) \neq \varnothing\) and we have
\begin{align}\label{eq:eqsvalfunct}
    f(\param) = F(\solp(\param), \param) 
    = \Lcal(\solpd(\param), \param).
\end{align}
The last equality in \eqref{eq:eqsvalfunct} is due to the fact that, for all \((\pvar, \dvar) \in \solpd(\param)\), one has by complementarity
\begin{equation*}
    \pvar \in \solp(\param),\; \transp{\dvar} C(\pvar, \param) = 0.
\end{equation*}
Thus, \(F(\solp(\param), \param) = F(\pvar, \param) = \Lcal(\pvar, \dvar, \param) = \Lcal(\solpd(\param), \param)\).
\end{remark}
To conclude, we will need the next Lemma which extends complementarity conditions for normal cones to time derivatives along smooth curves. 
\begin{lemma}
    \label{lem:complementarityCurves}
    Let $K \subset \Rm$ be a closed convex cone and
     $z \colon \R \to K$, $\lambda \colon \R \to K^\circ$ be both continuously differentiable such that $\lambda(t) \in N_K(z(t))$ for all $t$. Then $\langle \dot{\lambda}(t), z(t) \rangle = \langle \lambda(t), \dot z(t) \rangle = 0$ for all $t$.
\end{lemma}
\begin{proof}
    We have $\dot{\lambda}(t) \in T_{K^\circ}(\lambda(t))$ for all $t$. We have for any $\lambda \in K^{\circ}$, using Lemma~\myref{lem:normalConeCone} and \cite[Corollary 6.30]{rockafellar1998variational}
    \begin{align*}
        T_{K^\circ}(\lambda) = N_{K^\circ}(\lambda)^\circ = \{ u \in K^{\circ\circ}, \, \left\langle u, \lambda \right\rangle = 0\}^\circ.
    \end{align*}
    Since $K^{\circ\circ} = K$, $z \in K$ and $\langle z, \lambda \rangle = 0$, we deduce that for all $t$, $\langle z(t), \dot\lambda(t) \rangle \leq 0$. 
    We have similarly for all $t$, $\langle \dot z(t), \lambda(t) \rangle \leq 0$ and, since for all $t$, $\langle z(t), \lambda(t) \rangle = 0$, we obtain
    \begin{align*}
        0 = \frac{d}{dt}  \langle z(t), \lambda(t) \rangle = \langle z(t), \dot\lambda(t) \rangle + \langle \dot z(t), \lambda(t) \rangle,
    \end{align*}
    from which the result follows.
\end{proof}

%%%%%%%%%%%%%%%%%%%%%%%%%%%%%%%%%%%%%%%%%%%%%%%%%%%%%%%%%%%%%%%%
%%%%%%%%%%%%%%%%%%%%%%%%%%%%%%%%%%%%%%%%%%%%%%%%%%%%%%%%%%%%%%%%
\subsection{Parametric subdifferentiation of the value function}
\label{sect:parametricOptimality}

As discussed in the introduction, \cite[Theorem 10.13]{rockafellar1998variational} provides a formula for the subdifferential of \(f\), which has a strong connection with the adjoint formula in \eqref{form:asm}. Let us make this explicit in the context of Problem \eqref{prob:cpoCone}. We will consider, only in this section, the extended formulation in order to motivate our construction. Let
\begin{align*}
    \bar{F}\colon \Rn \times \Rq &\quad\to\quad \R \cup \{\infny\}\\
    (\pvar,\param) &\quad\mapsto\quad F(\pvar, \param) + \indicator_{K}(C(\pvar, \param)),
\end{align*}
where, for \(S \in \Rell\) with \(\ell \in \N\), \(\indicator_S\) is the indicator function of \(S\), i.e., \(\indicator_S(\vy) = 0\) if \(\vy \in S\) and \(\indicator_S(\vy) = \infny\) otherwise. Problem \eqref{prob:cpoCone} is equivalent to unconstrained partial minimization of $\bar{F}$ so that this fits the context of \cite[Theorem 10.13]{rockafellar1998variational}. 

Under Assumption~\myref{ass:standing}, the RCQ holds jointly in \((\pvar,\param)\) as it is the case for the MFCQ in Remark~\myref{rem:jointMFCQ}. Indeed, the first rows of the matrix \(\transp{\jac C(\pvar, \param)}\) are exactly those of \(\transp{\jac_{\pvar}\!C(\pvar, \param)}\). Therefore, the chain rule \cite[Theorem 10.6]{rockafellar1998variational} applies to the indicator in the definition of \(\bar{F}\). Note that the qualification condition is given in abstract form, it corresponds to Definition~\myref{def:mfcqCone} using the fact that for any $z \in K$, $N_K(z) = \{\lambda \in K^\circ,\, \left\langle \lambda, z \right\rangle = 0\}$ \cite[Corollary 23.5.4]{rockafellar1997convex} and \cite[Corollary 8.11, Definition 3.3]{rockafellar1998variational} since $K$ is a closed convex cone.

Combining this with the sum rule \cite[Exercise 10.10]{rockafellar1998variational}, we have, for all \((\pvar, \param) \in \graph R\),
\begin{align*}
    \clarke \bar{F}(\pvar, \param) &\subset \clarke F(\pvar, \param) +   \transp{\jac C(\pvar, \param)}N_{K}(C(\pvar, \param)),
\end{align*}
where the normal cone is the subdifferential of the indicator in the definition of \(\bar{F}\). In our case, the elements of the normal cones are given by the Lagrange multipliers \(\dvar\) in \(\solpd(\param)\); see Lemma~\myref{lem:normalConeCone}. Putting everything together, the parametric optimality condition, \cite[Theorem 10.13]{rockafellar1998variational}, ensures that, under Assumptions~\myref{ass:standing}, for all \(\param \in \Rq\), \(\clarke f(\param)\) is a subset of the following,
\begin{align}\label{eq:clarkesubdiffvf}
    & D_f(\param) \coloneq \conv \Big\{\vu \in \Rq : \exists\, (\pvar, \dvar) \in \solpd(\param),  (0, \vu) \in \clarke F(\pvar,\param) + \jac \transp{C(\pvar, \param)}\dvar \Big\}. 
\end{align}
Importantly, for all \(\param \in \Rq\) and for all \(\pvar \in \solp(\param)\), thanks to the parametric Fermat's rule \cite[Example 10.12]{rockafellar1998variational}, there exists \(\vu \in \Rq\) such that  \((0, \vu) \in \clarke \bar{F}(\pvar, \param)\); hence \(D_f(\param)\) is nonempty.

However, as already stated earlier, \(\partial^c f \neq D_f\) in general and the generalized adjoint formula produces artifacts; see Example~\myref{ex:failclarke}. We show in the coming section that, under Assumptions~\myref{ass:standing}, the set-valued map \(D_f\) is a conservative field for \(f\), providing a conservative adjoint formula.
%%%%%%%%%%%%%%%%%%%%%%%%%%%%%%%%%%%%%%%%%%%%%%%%%%%%%%%%%%%%%%%%
%%%%%%%%%%%%%%%%%%%%%%%%%%%%%%%%%%%%%%%%%%%%%%%%%%%%%%%%%%%%%%%%
\subsection{Conservative adjoint formula}

This is our main result. We consider a slightly more general version of \(D_f\) in \eqref{eq:clarkesubdiffvf}, allowing for a general conservative field for \(F\), beyond \(\partial^c F\). The motivation for the proposed form of \(D_f\) follows directly from the preceding discussion in Section~\myref{sect:parametricOptimality}.
\begin{theorem}[Main result]\label{theo:main}
    Let Assumptions~\myref{ass:standing} hold. Recall that \(F\) is a general locally Lipschitz function, not necessary differentiable, \(D_F\) is a conservative field for \(F\) and the value function \(f\colon \Rq \to \R\) is the mapping
     \begin{align*}
       \param \mapsto F(\solp(\param),\param).
     \end{align*}  
     Define \(\consv_f: \Rq \rightrightarrows \Rq\) by
    \begin{align*}
         \param \mapsto & \big\{ \vu \in \Rq : \exists\, (\pvar, \dvar) \in \solpd(\param), (0, \vu) \in  \conv \consv_{F}(\pvar, \param) +  \jac \transp{C(\pvar, \param)}\dvar \big\}.
    \end{align*} 
     Then 
     $\consv_f$  is a conservative field for $f$,  and in particular,
     \(\clarke f(\param) \subset \conv \consv_f(\param).\)
\end{theorem}

\begin{example}\label{ex:clarkecase} 
   The Clarke subdifferential is known to be conservative for definable functions \cite{davis2020stochastic}. Therefore, we can take the Clarke subdifferential of \(F\) as the conservative \(\consv_F\) for \(F\) in Theorem~\myref{theo:main} and everywhere \(\consv_F\) is used. 
   Optionally, for any \(\param \in \Rq\), we can replace \(\consv_{F(\cdot,\param)}\) by the Clarke partial subdifferential  \(\clarkep{\pvar}F(\cdot, \param)\), which is contained in the projection of the joint subdifferential \cite[Proposition 2.3.16]{clarke1983nonsmooth}. 
\end{example}

\begin{proof}[Proof of Theorem~\myref{theo:main}] 
    The proof follows similar lines as \cite[Theorem 12]{pauwels2023conservative}.
    Define the set-valued map \(\consv_{\Lcal}\colon \Rn \times \Rm \times \Rq \rightrightarrows \Rn \times \Rq\) as
    \begin{align*}
        \consv_{\Lcal}(\vy,\param) &= \consv_{\Lcal}(\pvar, \dvar, \param) \coloneq \consv_{F}(\pvar, \param) +  \left(\jac_{\pvar}\!\transp{C(\pvar, \param)}\dvar , \jac_{\param}\!\transp{C(\pvar, \param)}\dvar \right),
    \end{align*}
    where we use the shorthand notation \(\vy = (\pvar, \dvar)\). Note that \(\consv_{\Lcal}\) is a partial conservative field for \(\Lcal\), which ignores the variations in \(\dvar \in \Rm\). This will be possible thanks to complementarity and the structure of $\Lcal$.
    
    As \(\consv_\Lcal\), \(\solp\) and \(\solpd\) can be described by first-order formulas involving only definable objects, they are definable; see \cite[Theorem 1.13]{coste2000introduction}. 
    By definition, the definable conservative field \(\consv_{F}\) is graph closed, compact valued and locally bounded.
    By the parametric Fermat's rule in \cite[Example 10.12]{rockafellar1998variational}, \(\consv_\Lcal\) has nonempty values.
    It is also immediate that \(\consv_\Lcal\) is definable, graph closed, locally bounded, and compact valued by Proposition~\myref{prop:locboundp}.
    
    According to \cite[Theorem 2 and Lemma 6]{pauwels2023conservative}, to finish the proof, it remains to show that \(\consv_f\) has the chain rule property along continuously differentiable  and definable curves. So let \(\vt \mapsto \param(\vt)\) be a continuously differentiable  and definable curve in \(\Rq\). The following definable selections are possible using \cite[paragraph 4.5]{van1996geometric}. 

    We fix for now \(\vt \mapsto \vy(\vt) \coloneq (\pvar(\vt), \dvar(\vt))\) a definable selection in \(\vt \rightrightarrows \solpd(\param(\vt))\), with \(\vt \mapsto \pvar(\vt)\) a definable selection in \(\vt \rightrightarrows \solp(\param(\vt))\) and \(\vt \mapsto \dvar(\vt)\) a definable selection in \(\vt \rightrightarrows \kkt(\pvar(\vt), \param(t))\).
    Let \(\vt \mapsto T(\vt) \coloneq (\vv(\vt), \vu(\vt))\) be a definable selection in \(\vt \rightrightarrows \consv_{\Lcal}(\vy(\vt),\param(\vt))\).
    There is a definable selection \(\vt \mapsto (\vz(\vt), \vw(\vt))\) in \(\vt \rightrightarrows \consv_{F}(\pvar(\vt), \param(\vt))\) such that for all $\vt$ 
    \begin{align*}
        \vv(\vt) &= \vz(t) + \transp{\jac_{\pvar}\!C(\pvar(\vt), \param(\vt))}\dvar(\vt) \\
        \vu(\vt) &= \vw(t) + \transp{\jac_{\param}\!C(\pvar(\vt), \param(\vt))}\dvar(\vt).
    \end{align*}
    Since definable curves are piecewise continuously differentiable, all considered selections as well as 
    \[\vt \mapsto f(\param(\vt)) = F(\pvar(\vt), \param(\vt)) = \Lcal(\pvar(\vt), \dvar(\vt), \param(\vt)),\]
    are differentiable everywhere except at finitely many points, say at \(\vt_1, \vt_2, \ldots, \vt_M\). By Proposition~\myref{prop:locboundp}, \(\solpd\) is locally bounded. It follows from the monotonicity Lemma \cite[Section 2.1]{coste2000introduction}that \(\vt \mapsto \vy(\vt) = (\pvar(\vt), \dvar(\vt))\) has left and right limits everywhere and can be extended to an absolutely continuous curve on \([t_{\ell-1}, t_{\ell}]\), for any \(\ell \in \{2, \dots, M\}\). 
    
    Using the equalities in \eqref{eq:eqsvalfunct}, the continuity of \(f\) given by Lemma~\myref{lem:fcont}, and the continuity of \(F\) and \(C\), we have
    \begin{align}
        f(\param(\vt_{\ell})) - f(\param(\vt_{\ell-1})) &= \lim_{\vt \uparrow \vt_{\ell}} \Lcal(\pvar(\vt), \dvar(\vt), \param(\vt)) - \lim_{\vt \downarrow \vt_{\ell-1}} \Lcal(\pvar(\vt), \dvar(\vt), \param(\vt)) \nonumber \\
        &= F(\pvar(\vt_{\ell}^{-}), \param(\vt_{\ell})) -  F(\pvar(\vt_{\ell-1}^{+}), \param(\vt_{\ell-1})) \nonumber \\
        &\qquad + \transp{\dvar(\vt_{\ell}^{-})} C(\pvar(\vt_{\ell}^{-}), \param(\vt_{\ell}))  - \transp{\dvar(\vt_{\ell-1}^{+})} C(\pvar(\vt_{\ell-1}^{+}), \param(\vt_{\ell-1})) \nonumber \\
        &= \int_{\vt_{\ell-1}}^{\vt_{\ell}} \frac{d}{dt} F(\pvar(\vt), \param(\vt))\,dt + \int_{\vt_{\ell-1}}^{\vt_{\ell}} \frac{d}{dt}\left[\transp{\dvar(\vt)} C(\pvar(\vt), \param(\vt))\right]dt. \nonumber 
    \end{align}
    Using the fact that \(\consv_F\) is conservative for \(F\) in the first integral and differentiating in the others give
    \begin{align}
        &f(\param(\vt_{\ell})) - f(\param(\vt_{\ell-1})) = \int_{\vt_{\ell-1}}^{\vt_{\ell}} \big[ \scalarp{\dot{\param}(\vt), \vw(t)} + \scalarp{\dot{\pvar}(\vt), \vz(\vt)} \big] dt + \int_{\vt_{\ell-1}}^{\vt_{\ell}} \big[ \scalarp{\dot{\dvar}(\vt), C(\pvar(\vt), \param(\vt))} \big] dt \nonumber\\ 
        &+ \int_{\vt_{\ell-1}}^{\vt_{\ell}} \big[ \scalarp{\dvar(\vt), \jac_{\param}\!C(\pvar(\vt), \param(\vt))\dot{\param}(\vt)} + \scalarp{\dvar(\vt), \jac_{\pvar}\!C(\pvar(\vt), \param(\vt))\dot{\pvar}(\vt)} \big] dt.
        \label{eq:intchainrule}
    \end{align}    
    We note that for all $t, C(x(t),\param(t)) \in K$ and $\lambda(t) \in N_K(C(x(t), \param(t))$ thanks to Lemma~\myref{lem:normalConeCone}. Since both curves are continuously differentiable on \((t_{\ell-1}, t_{\ell})\), Lemma~\myref{lem:complementarityCurves} ensures that  
    \begin{equation}\label{eq:nullprod}
     t \mapsto \scalarp{\dot{\dvar}(\vt), C(\pvar(\vt), \param(\vt))} \equiv 0 \; \text{ on } (t_{\ell-1}, t_{\ell}).
    \end{equation}
    Using \eqref{eq:nullprod}, removing the finitely many discontinuous points and putting everything together, Equation~\eqref{eq:intchainrule} gives, for almost every \(\vt\),
    \begin{equation}
        \frac{d}{dt} f(\param(\vt)) = \scalarp{\dot{\pvar}(\vt), \vv(\vt)} + \scalarp{\dot{\param}(\vt), \vu(\vt)}.
        \label{eq:chainRuleSmallf}
    \end{equation}
    Since \(\consv_{\Lcal}\) is definable and compact valued, Lemma~\myref{lem:edlem8} says that there is a sequence of definable selections, as \(v\) and \(u\), that is dense in \(\consv_{\Lcal}( \vy(\vt), \param(\vt))\), for all
    \(t\). Each element of the sequence satisfies \eqref{eq:chainRuleSmallf} for almost all $t$. Combining the closure in Lemma~\myref{lem:edlem8}, with the fact that a countable union of Lebesgue negligible set is Lebesgue negligible, we obtain, for almost every \(\vt\),
    \begin{equation*}
        \frac{d}{dt} f(\param(\vt)) = \scalarp{\dot{\pvar}(\vt), \vv} + \scalarp{\dot{\param}(\vt), \vu} \quad \forall\, (\vv,\vu) \in \consv_{\Lcal}( \vy(\vt), \param(\vt)).
    \end{equation*}
    Then, using the fact that \((0,u) \in \consv_{\Lcal}( \vy(\vt), \param(\vt)) \) for all \(t\), we have, for almost every \(\vt\), 
    \begin{equation}\label{eq:uzeroarb}
        \frac{d}{dt} f(\param(\vt)) = \scalarp{\dot{\param}(\vt), \vu} \quad \forall\, \vu \in \Rq \text{ such that } (0, \vu) \in \consv_{\Lcal}( \vy(\vt), \param(\vt)).
    \end{equation}
    Note that the \(t \mapsto \vy(\vt)\) has been fixed until now. 
    Thanks to the density result in Lemma~\myref{lem:edlem9}, we know that there exists a sequence of definable selections 
    \(\Big(\underbrace{(\pvar_{\ell}, \dvar_{\ell})}_{=:\vy_{\ell}}\Big)_{\ell \in \N}\)
    such that, for all \(\ell \in \N\) and all \(\vt\), \(\vy_{\ell}(\vt) \in \solpd(\param(\vt))\) and such that, for all \(\vt\),
    \begin{align*}
        \consv_f(\param(\vt)) &= \big\{\vu \in \Rq : \exists\, \vy \in \solpd(\param(\vt)), (0, \vu)  \in \consv_{\Lcal}(\vy, \param(\vt))\big\} \\
        &=\cl\bigcup_{\ell\in\N}\big\{\vu \in \Rq : (0, \vu) \in \consv_{\Lcal}(\vy_{\ell}(\vt), \param(\vt))\big\}.
    \end{align*}
    For each $\ell \in \N$, $y_\ell$  satisfies \eqref{eq:uzeroarb} which applies to an arbitrary selection \(\vy(\vt) = (\pvar(\vt), \dvar(\vt)) \in \solpd(\param(\vt))\). Using the same density argument, it follows, that, for almost every \(\vt\),
    \begin{equation*}
        \frac{d}{dt} f(\param(\vt)) = \scalarp{\dot{\param}(\vt), \vu} \quad \forall\, \vu \in \consv_f(\param(\vt)).
    \end{equation*}
    This concludes the proof.
\end{proof}

%%%%%%%%%%%%%%%%%%%%%%%%%%%%%%%%%%%%%%%%%%%%%%%%%%%%%%%%%%%%%%%%
%%%%%%%%%%%%%%%%%%%%%%%%%%%%%%%%%%%%%%%%%%%%%%%%%%%%%%%%%%%%%%%%

\subsection{Gradient-like algorithms using the adjoint method}\label{sect:implicopt}
The main result of the previous section ensures that the adjoint method provides a selection in a conservative field for $f$. General conservative fields can be used as a first order oracle to optimize $f$ \cite{bolte2021conservative}. We describe below how the adjoint method yields an optimization method through a vanilla small-step gradient method   for the problem:
    $$\min_{\param \in \Rq} f(\param).$$

\begin{proposition}[Small step method]
    \label{cor:smallStepMethod}
    Under the assumptions of Theorem~\myref{theo:main}, consider the recursion
    \begin{align*}
        \param_{k+1} = \param_k - s\alpha_k \emph{ASM}(\param_k),
    \end{align*}
    where $\emph{ASM}$ is a selection of $\consv_f$, $s>0$, $\alpha_k > 0$ for all $k$ and $\alpha_k \to 0$ as $k \to \infty$.
    For almost all $\param_0 \in \R^q$ and all but finitely many $s > 0$, if the sequence $(\param_k)_{k \in \N}$ is bounded, then any accumulation point $\bar{\param}$ of $(\param_k)_{k\in\N}$ is Clarke critical, i.e., $0 \in \partial^c f(\bar{\param})$, where $\partial^c f(\bar{\param})$ is the Clarke generalized gradient of $f$ at $\bar{\param}$.
\end{proposition}
The proof is not provided since it is a direct consequence  of \cite[Theorem 9]{bolte2021conservative}, \cite[Theorem 7]{bolte2023subgradient}, see the proof of \cite[Theorem 4]{pauwels2023conservative}.

Proposition~\myref{cor:smallStepMethod} is related to the basic gradient method. The main device behind this result is continuous time Lyapunov decrease. Similar results hold for the heavy ball method \cite{le2024nonsmooth}, subgradient methods with Bregman distance \cite{ding2025stochastic} or adaptive methods such as Adam \cite{xiao2024adam} under appropriate assumptions.

%%%%%%%%%%%%%%%%%%%%%%%%%%%%%%%%%%%%%%%%%%%%%%%%%%%%%%%%%%%%%%%%
%%%%%%%%%%%%%%%%%%%%%%%%%%%%%%%%%%%%%%%%%%%%%%%%%%%%%%%%%%%%%%%%
\section{Failure of ASM without definability}\label{sect:failure}
Let us illustrate the importance of the definability assumptions in Problem \eqref{prob:cpoCone}. For this we demonstrate that the local Lipschitz continuity, path differentiability and even differentiability of the objective function \(F\), together with the continuous differentiability of the constraint functions \(G\) and \(H\), are not sufficient for \(D_f\) in Theorem~\myref{theo:main} to be a conservative field for \(f\). The restriction to a subclass of functions with additional properties is required. As already stated, in this work, we elected for definable functions that arise in many relevant applications.

Before presenting the result, we describe a fractal-type set that will be central to the proof. Indeed, we define the set \(S \subset \R^2\) as \(S \coloneq \cap_{\ell \in \N} S_{\ell}\), where for each \(\ell \in \N\), \(S_{\ell}\) is the union of \(4^{\ell}\) squares of side length \(1 / 4^{\ell}\) each and \((S_{\ell})_{\ell \in \N}\) form a nested decreasing sequence for the inclusion. In addition, the set \(S\) has the following properties:
\begin{enumerate}[label=\rm{(\roman*)}, ref=\rm{\roman*}, align=left, labelwidth=*, leftmargin=*, labelindent=0.5em, labelsep=-0.4em]\label{def:setcounterexample}

    \item \label{def:d1} \(S\) is nonempty and closed with empty interior, and is included in \([0,1] \times [0,1]\); hence it is compact; 
            
    \item \label{def:d3} \(S\) projected on the first axis gives \([0, 1] \times \{0\}\), and projected on the second axis it gives \(\{0\}\times [0, 1]\).
    
\end{enumerate}
Such a set \(S\) exists and can be constructed; see \cite[Section 4.1]{pauwels2023conservative} for all the details and all its properties. An illustration of the construction \cite[Figure 1]{pauwels2023conservative} is reproduced exactly in Figure~\myref{fig:fractal}.

\begin{figure}[ht]
				\centering
				\includegraphics[width=.8\textwidth]{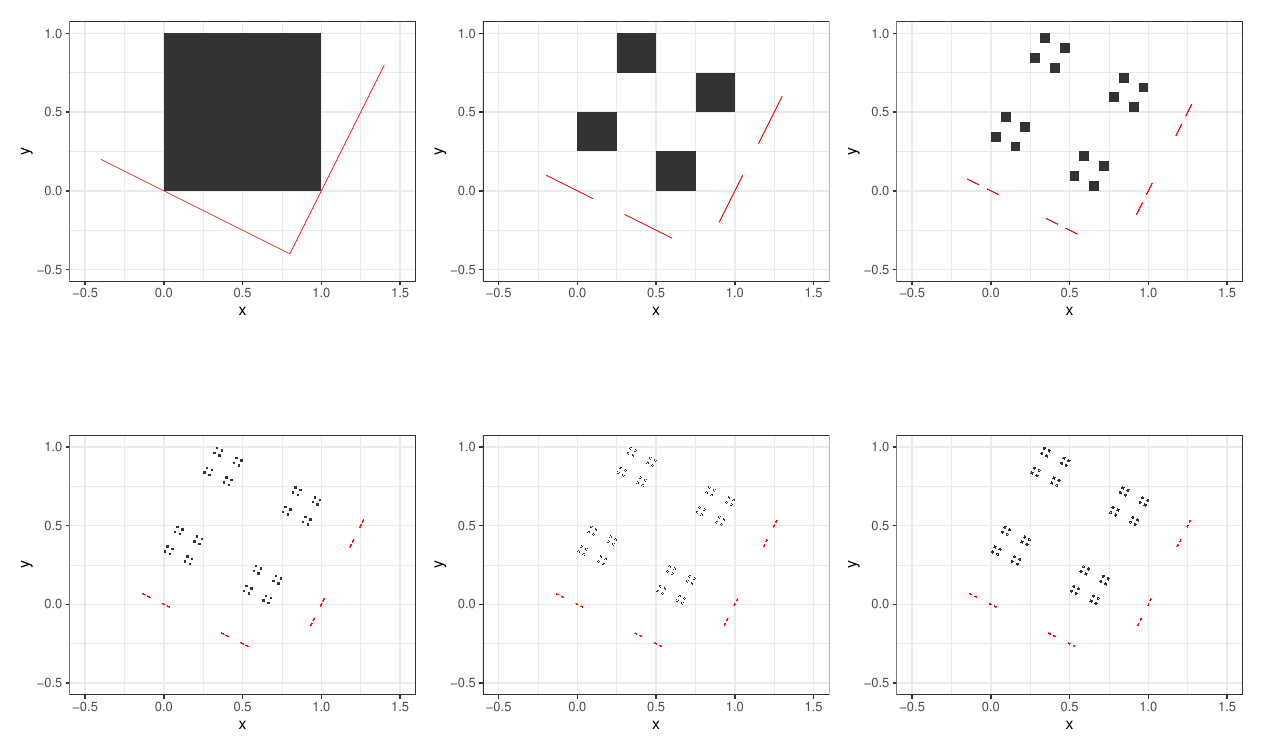}
				\caption{The fractal construction in \cite[Section 4.1]{pauwels2023conservative}. They start with the closed unit square in black. It is divided into 16 squares, each with a side length equal to one fourth of the original square’s side length. Only four of them at specific positions are kept and the others dropped. This process is repeated recursively on each square ad infinitum. The additional red lines represent projection of these sets on rotated axes. Considering \(S_{\ell}\), \(\ell \in \N\), the set obtained after \(\ell\) steps (\(S_{0}\) is the original square). We have that \(S_{\ell+1} \subset S_{\ell}\) for all \(\ell\). Then, let \(S = \cap_{\ell \in \N} S_{\ell}\), which is closed. The projection of \(S\) on each axes are full segments. Furthermore, in the limit, both projections on rotated axes are Cantor sets of zero measure.}
				\label{fig:fractal}
\end{figure}

\begin{claim}\label{theo:counterex} 
    There exists a locally Lipschitz objective function \(F \colon \R^2 \to \R\) that 
    \begin{enumerate}[label=\rm{(\roman*)}, ref=\rm{\roman*}, align=left, labelwidth=*, leftmargin=*, labelindent=0.5em, labelsep=-0.4em]
        \item is not definable in any o-minimal structure, 
        
        \item is differentiable and path differentiable on \(\R^2\), continuously differentiable on \(\R^2 \setminus S\), and 
        
        \item is  such that the set-valued map \(D_f\colon \R \rightrightarrows \R\) given by
        \begin{align*}
            \param &\mapsto \big\{\vu \in \R : \exists\, (\pvar, \dvar) \in \solpd(\param), (0, \vu) \in  \clarke F(\pvar, \param) +  \left(\dvar \nabla_{\!x} G(\pvar, \param), 0 \right)\big\}
        \end{align*}     
        is not a conservative field for the value function 
        \(f\colon \R \to \R\) defined by
        \begin{align*}
            \param \mapsto \min_{\pvar \in \R} \{F(\pvar, \param): G(\pvar, \param) \leq 0\},
        \end{align*}
        where G is the continuously differentiable function \((\pvar, \param) \mapsto \pvar^2 - 4\).
    \end{enumerate}
    Moreover, \(F\) and \(G\) are such that Assumptions~\myeqhyperrefrange{ass:standing}{ass:a2}{ass:a2}{ass:a3} hold.
\end{claim}

\begin{proof}
    By Claim~\myref{lem:fnctcounterex}, there exists a differentiable, locally Lipschitz and path differentiable function \(F\colon 
    \R^2 \to \R\) such that \(F(\pvar, \param)= 0\) for every \((\pvar, \param) \in S\), and \(F(\pvar, \param) > 0\) for every \((\pvar, \param) \in \R^2 \setminus S\). In addition, \(F\) is continuously differentiable everywhere except on \(S\) with 
    \begin{align}\label{eq:inclusionclarkef}
    \Bcal_c(0,1) \subset \clarke F(\pvar, \param)\quad \forall\, (\pvar, \param) \in S.
    \end{align}
    Given the fact that \(\displaystyle \argmin_{(\pvar, \param) \in \R^2} F(\pvar, \param) = S\) is not definable, \(F\) is not definable either.

    The set-valued map of feasible point \(R\colon \R \to \R\), defined by
    \begin{align*}
        \param \mapsto \{\pvar \in \R : G(\pvar, \param) \leq 0\},
    \end{align*}
    has the value \(R(\param) = [-2,2]\) for any \(\param \in \R\). Given \(\param \in \R\), the constraint is only active at \((-2,\param)\) and \((2,\param)\). Additionally, \(\nabla_{\!x} G(-2, \param) \neq 0\) and \(\nabla_{\!x} G(2, \param) \neq 0\); thus the MFCQ holds at any \((\pvar, \param) \in \graph R\), i.e., Assumption~\myhyperref{ass:standing}{ass:a2} is satisfied.
    
    Since, \(F\) is continuous and for every \(\param \in \R\), \(R(\param) = [-2,2]\), the set-valued map of primal solutions \(\solp \colon \R \to \R\), defined by
    \begin{align*}
        \param \mapsto \argmin \{F(\pvar, \param) : \pvar \in R(\param)\},
    \end{align*}
    has nonempty values on \(\R\), and \(\solp(\param) \subset [-2,2]\) for any \(\param \in \R\), which means that \(\solp\) is uniformly bounded on \(\R\); hence Assumption~\myhyperref{ass:standing}{ass:a3} is verified. 

    Thanks to Definition~\myhyperref{def:setcounterexample}{def:d3}, for every \(\param \in [0,1]\), there exists \(\pvar \in [0,1]\) such that \((\pvar, \param) \in S\). It follows that the value function is zero on \([0,1]\), i.e,
    \begin{equation*}
    \forall\,\param \in [0,1],\quad f(\param) = 0.
    \end{equation*}

    Let \(\param \in [0,1]\) and let \(\bar{\pvar} \in \solp(\param)\). As \(\bar{\pvar} \in [0,1]\) and \(G\) is therefore not active at \((\bar{\pvar}, \param)\), we have \( (\bar{\pvar}, 0) \in \solpd(\param) \). It follows that
    \begin{align*}
        D_f(\param) &= \big\{\vu \in \R : \exists\, (\pvar, \dvar) \in \solpd(\param), (0, \vu) \in  \clarke F(\pvar, \param) +  \left(\lambda \nabla_{\!x} G(\pvar, \param), 0 \right)\big\} \\
        &\supset \big\{\vu \in \R : (0, \vu) \in  \clarke F(\bar{\pvar}, \param)\big\} \\
        &\supset \big\{\vu \in \R : (0, \vu) \in  \Bcal_c(0,1)\big\} \\
        &= [-1,1],
    \end{align*}
    where we use \eqref{eq:inclusionclarkef} to get the second inclusion.
    
    Consider the absolutely continuous curve \(\gamma \colon [0,1] \to \R \) defined as \(t \mapsto t\). Now we suppose, for the sake of contradiction, that \(f\) has the chain rule property for \(D_f\). Then, for almost all \(t \in [0,1]\), and for any \(\vu \in D_f(\gamma(t))\), in particular any \(\vu \in [-1,1]\),
    \begin{align*}
        \frac{d}{dt} f(\gamma(t)) = \scalarp{u, \dot{\gamma}(t)} = \scalarp{u, 1} = u.
    \end{align*}
    This is impossible. So \(f\) does not have the chain rule property for \(D_f\), i.e, \(D_f\) is not conservative for \(f\).
\end{proof}

%%%%%%%%%%%%%%%%%%%%%%%%%%%%%%%%%%%%%%%%%%%%%%%%%%%%%%%%%%%%%%%%%%%%%%%%%%%%%%%%%%%%
%%%%%%%%%%%%%%%%%%%%%%%%%%%%%%%%%%%% Appendix %%%%%%%%%%%%%%%%%%%%%%%%%%%%%
%%%%%%%%%%%%%%%%%%%%%%%%%%%%%%%%%%%%%%%%%%%%%%%%%%%%%%%%%%%%%%%%%%%%%%%%%%%%%%%%%%%
\appendix 

\section{Technical results}
The following lemma is an immediate consequence of the stronger \emph{Aubin property} result given in \cite[Corollary 4.4]{mordukhovich1994lipschitzian}. We give its quick derivation in the proof just for completeness.
\begin{lemma}\label{lem:rinnersemicontinuous}
    Under Assumption~\myref{ass:standing}, the set-valued map of feasible points \(R\) is inner semicontinuous, i.e., for all \((\pvar^{*}, \param^{*}) \in \Rn \times \Rq\) and \((\param_k)_{k \in \N} \subset \Rq\) such that \(\pvar^{*} \in R(\param^{*})\) and \(\param_k \xrightarrow[k \to \infny]{} \param^{*}\), there exists \((\pvar_k)_{k \in \N} \subset \Rn\) such that \(\pvar_k \xrightarrow[k \to \infny]{} \pvar^*\) and \(\pvar_k \in R(\param_k)\) for any \(k \in \N\).
\end{lemma}
\begin{proof}
    Let \((\pvar^{*}, \param^{*}) \in \Rn \times \Rq\) such that \(\pvar^{*} \in R(\param^{*})\), and \((\param_k)_{k \in \N} \subset \Rq\) such that \(\param_k \xrightarrow[k \to \infny]{} \param^{*}\). Since \((\pvar^{*}, \param^{*})\) is RCQ-qualified as in Definition~\myref{def:mfcqCone}, \cite[Corollary 4.4, (4.16)]{mordukhovich1994lipschitzian} ensures that \(R\) has the {Aubin property}, or equivalently, is \emph{pseudo-Lipschitzian}, around \((\pvar^{*}, \param^{*})\). That means there exist a neighborhood \(\Ncal\) of \(\pvar^{*}\), a neighborhood \(\Ucal\) of \(\param^{*}\) and \(M \geq 0\) such that
    \begin{align*}
        R(\param_1) \cap \Ncal \subset R(\param_2) + M \norm{\param_1-\param_2}\Bcal_{c}(0,1) \qquad \forall\, \param_1, \param_2 \in \Ucal.
    \end{align*}
    In particular, for any \(k \geq \bar{k}\), with \(\bar{k}\) large enough, 
    \begin{align*}
       \{\pvar^{*}\} \subset R(\param^{*}) \cap \Ncal \subset R(\param_k) + M \norm{\param^{*} - \param_k}\Bcal_{c}(0,1).
    \end{align*}
    Therefore, for any \(k \geq \bar{k}\) , there exists \(\pvar_k \in R(\param_k)\) and \(\vz_k \in \Bcal_{c}(0,1)\) such that 
    \[\pvar^{*} = \pvar_k + M \norm{\param^{*} - \param_k}\vz_k.\]
    It follows that
    \begin{align*}
        \norm{\pvar^{*} - \pvar_k} \leq M \norm{\param^{*} - \param_k}\text{ and } \lim_{k \to \infny} \norm{\pvar^{*} - \pvar_k} = 0.
    \end{align*}
    So \(\pvar_k \xrightarrow[k \to \infny]{} \pvar^{*}\), which proves that \(R\) is inner semicontinuous.
\end{proof}

\begin{lemma}\label{lem:sclosedgraph}
    Under Assumption~\myref{ass:standing}, the set-valued map of primal solutions \(\solp\) has a closed graph.
\end{lemma}
\begin{proof}
    Let \((\param_k)_{k \in \N} \subset \Rq\),  \((\pvar_k)_{k \in \N} \subset \Rn\) and \((\pvar^{*}, \param^{*}) \in \Rn \times \Rq\) such that \(\pvar_k \in \solp(\param_k)\) for all \(k \in \N\), \(\pvar_k \xrightarrow[k \to \infny]{} \pvar^{*}\) and \(\param_k \xrightarrow[k \to \infny]{} \param^{*}\). Using continuity of $C$ and closedness of $K$, we have that \(\pvar^{*} \in R(\param^{*})\).
    
    Take \(\vy \in R(\param^{*})\). Since \(R\) is inner semicontinuous thanks to Lemma~\myref{lem:rinnersemicontinuous}, there exists \((\vy_k)_{k \in \N} \subset \Rn\) such that \(\vy_k \xrightarrow[k \to \infny]{} \vy\) and \(\vy_k \in R(\param_k)\) for any \(k \in \N\). It follows that, for every \(k \in \N\),
    \begin{equation*}
        F(\pvar_k, \param_k) \leq F(\vy_k, \param_k),
    \end{equation*}
    which implies that
    \begin{equation*}
       \lim_{k \to \infny} F(\pvar_k, \param_k) \leq \lim_{k \to \infny} F(\vy_k, \param_k),
    \end{equation*}
    and, by continuity of \(F\),
    \begin{equation*}
        F(\pvar^{*}, \param^{*}) \leq F(\vy, \param^{*}).
    \end{equation*}
    We can conclude that \(\pvar^{*} \in \solp(\param^{*})\) and that \(\solp\) is graph closed.
\end{proof}
The following lemma can also be obtained through the Berge's maximum theorem. We provide a self contained proof which fits Assumption~\myref{ass:standing}
\begin{lemma}\label{lem:fcont}
    Under Assumption~\myref{ass:standing}, the value function \(f\), defined in \eqref{prob:cpoCone}, is continuous on \(\Rq\).
\end{lemma}
\begin{proof}
    Let \((\param_k)_{k \in \N} \subset \Rq\) and \( \param^{*} \in \Rq\) such that \(\param_k \xrightarrow[k \to \infny]{} \param^{*}\). Take \(\pvar_k \in \solp(\param_k)\) for each \(k \in \N\). Since \(\solp\) is locally bounded, there exists a neighborhood \(\Ncal\) of \(\param^{*}\) and \(r > 0\) such that
     \begin{equation*}
        \bigcup_{\param \in \Ncal} \solp(\param) \subset \Bcal_{c}(0, r).
     \end{equation*} 
     For \(\bar{k}\) large enough, we have \(\param_k \in \Ncal\) and \(\pvar_k \in \solp(\param_k) \subset \Bcal_{c}(0, r)\) whenever \(k \geq \bar{k}\). We then get that \((\pvar_k)_{k \in \N}\) is bounded. Therefore \((f(\param_k))_{k \in \N} = (F(\pvar_k, \param_k))_{k \in \N}\) is bounded by continuity of \(F\). 
     
     Let \(f^* \in \R\) be a limit point of \((f(\param_k))_{k \in \N}\). Thus there exists \((f(\param_{k_{s}}))_{s \in \N}\) such that \(f(\param_{k_{s}}) \xrightarrow[s \to \infny]{} f^* \). The sequence \((\param_{k_{s}})_{s \in \N}\), being a subsequence of \((\param_k)_{k \in \N}\), converges to \(\param^{*}\). The corresponding sequence \((\pvar_{k_{s}})_{s \in \N}\) is bounded by the previous paragraph. So, we can extract a convergent subsequence \((\pvar_{k_{s_{\ell}}})_{\ell \in \N}\) and let \(\pvar^{*}\) be its limit. Since \(\solp\) has a closed graph by Lemma~\myref{lem:sclosedgraph}, \(\pvar^{*} \in \solp(\param^{*})\). It follows that
     \begin{equation*}
         f^* = \lim_{s \to \infny} f(\param_{k_{s}}) =  \lim_{\ell \to \infny} f(\param_{k_{s_{\ell}}}) = \lim_{\ell \to \infny} F( \pvar_{k_{s_{\ell}}}, \param_{k_{s_{\ell}}}) = F(\pvar^{*}, \param^{*}) =  f(\param^{*}).
     \end{equation*} 
     As \(f^*\) was arbitrary, we get that \(f(\param^{*})\) is the only accumulation point of the bounded sequence \((f(\param_k))_{k \in \N}\). Consequently, \(f(\param_{k}) \xrightarrow[k \to \infny]{} f(\param^{*})\) and \(f\) is continuous.
\end{proof}

The following two lemmas are definable versions of Castaing representation theorem.
\begin{lemma}[{\cite[Lemma 8]{pauwels2023conservative}}]\label{lem:edlem8} Let \(J\colon \Rs \rightrightarrows \Rr\) be a compact-valued definable set-valued map with nonempty values. Then there exists a sequence of definable selections \((\vV_{\ell})_{{\ell} \in \N}\) for \(J\) such that, for any \(\pvar \in \Rs\),
\begin{equation*}
    J(\pvar) = \cl\{\vV_{\ell}(\pvar):\ell \in \N\}.
\end{equation*}
\end{lemma}

\begin{lemma}[{\cite[Lemma 9]{pauwels2023conservative}}]\label{lem:edlem9} 
    Let Assumptions~\myref{ass:standing} hold. Let \(\consv_{F}\) be a definable conservative field for \(F\). Then there exists a sequence of definable selections \(\left\{(\pvar_{\ell}, \dvar_{\ell})\right\}_{{\ell} \in \N}\) such that, for any \({\ell}\in\N\) and \(\param \in \Rq\), \((\pvar_{\ell}(\param), \dvar_{\ell}(\param)) \in \solpd(\param)\) and such that, for all \(\param \in \Rq\),
    \begin{align*}
        &\Big\{\vu \in \Rq : \exists\, (\pvar, \dvar) \in \solpd(\param), (0, \vu) \in  \conv\consv_{F}(\pvar, \param)  + \left(\jac_{\pvar}\!\transp{C(\pvar, \param)}\dvar,\jac_{\param}\!\transp{C(\pvar, \param)}\dvar \right) \Big\} \\
        &= \cl \big\{\vu \in \Rq :\exists\, {\ell}\in\N, (0, \vu) \in  \conv\consv_{F}(\pvar_{\ell}(\param), \param) \\
        &\qquad \quad + \big(\jac_{\pvar}\!\transp{C(\pvar_{\ell}(\param), \param)}\dvar_{\ell}(\param) ,\jac_{\param}\!\transp{C(\pvar_{\ell}(\param), \param)}\dvar_{\ell}(\param) \big) \big\}.
    \end{align*}
\end{lemma}

The following results are technical elements for the construction of the counter example of Section~\myref{sect:failure}.
\begin{lemma}\label{lem:subseqtocantor}
    Let \(A \subset \R^{M}\) be closed with empty interior. Then \(A\) is the set of accumulation points of a sequence \((\vz_k)_{k\in\N} \subset \R^{M} \setminus A\) with \(\vz_s \neq \vz_r\) for all \(s \neq r\).
\end{lemma}
\begin{proof} 
    Since every subset of \(\R^{M}\) is itself separable, there exists a sequence \((c_r)_{r\in\N} \subset A\) which is dense in \(A\). We know that \(\R^{M} \setminus A\) is dense in \(\R^{M}\). Therefore, for every \(\ell \in \N\) and \(s \in \{1,\ldots,\ell\}\), we can take 
    \[
    \vw_{\ell,s} \in (\R^{M}\setminus A) \cap \Bcal(c_s, 1/\ell)
    \]
    such that
    \[ 
    \vw_{\ell,s} \notin \{\vw_{r,t}: r,t \in \N, r < \ell, \text{ and } 1\leq t \leq r\} \cup \{\vw_{r,t}: r,t \in \N, r = \ell, \text{ and } 1\leq t < s\}.
    \]
    We set \((\vz_k)_{k \in \N}\) as an ordered enumeration of 
    \(
    \{\vw_{1,1}, \vw_{2,1}, \vw_{2,2}, \vw_{3,1}, \vw_{3,2}, \vw_{3,3}, \vw_{4,1}, \ldots\}.
    \)
    
Let \(\bar{\vz} \in A\). Then there exists \((c_{r_s})_{s\in\N} \subset A\) such \(c_{r_s} \xrightarrow[s \to \infny]{} \bar{\vz}\). Let \(\varepsilon > 0\) and \(N \in \N\). There exists \(s \in \N\) such that \(\metricd{c_{r_s}}{\bar{\vz}} < \varepsilon /2\). Let \(L \in \N\) such that \(1/L < \varepsilon/2\) and \(L > r_s\). We have \(\vw_{\ell,r_s} \in \Bcal(c_{r_s}, 1/\ell)\) for all \(\ell \geq L\). It follows that \(\metricd{\vw_{\ell,r_s}}{\bar{\vz}} < \varepsilon\) for all \(\ell \geq L\). Let \(\ell \geq \max\{L,N\}\). By definition, there exists \(k \geq N\) such that \(\vz_k = \vw_{\ell, r_s}\). Then \(\metricd{\vz_k}{\bar{\vz}} < \varepsilon\) and \(\bar{z}\) is an accumulation point of \((\vz_k)_{k \in \N}\).

Let \(\bar{\vz} \not\in A\). Since \(A\) is closed, we have that \(\dist(\bar{\vz}, A) > 0\), with 
\[
\dist(\bar{\vz}, A) = \inf\{\metricd{\bar{\vz}}{a}: a \in A\}.
\]
Let \(L \in \N\) such that \(1/L < \dist(\bar{\vz}, A) / 2\). So for every \(\ell \geq L\), for all \(s \in \{1, \ldots, \ell\}\), \(\metricd{c_s}{\vw_{\ell,s}} < 1/\ell \leq 1/L\) . Suppose that there exist \(\ell \geq L\) and \(s \in \{1,\ldots,\ell\}\) such that \(\metricd{\vw_{\ell,s}}{\bar{\vz}} < 1/L\). It follows that 
\[
2/L < \dist(\bar{\vz}, A) \leq \metricd{\bar{\vz}}{c_s} \leq \metricd{\bar{\vz}}{\vw_{\ell,s}} + \metricd{\vw_{\ell,s}}{c_s} < 1/L + 1/L = 2/L.\] This is a contradiction. So for every \(\ell \geq L\), \(\metricd{\bar{\vz}}{\vw_{\ell,s}} \geq 1/L\) for all \(s \in \{1, \ldots, \ell\}\). That means that there exists \(\bar{k} \in \N\) such that for any \(k \geq \bar{k}\), \(\metricd{\vz_k}{\bar{\vz}} \geq 1/L\). Thus \(\bar{\vz}\) is not an accumulation point of \((\vz_k)_{k \in \N}\).

We can conclude that the accumulation points of \((\vz_k)_{k \in \N}\) are exactly \(A\).
\end{proof}

\begin{claim}\label{lem:fnctcounterex}
    Let \(S \subset \Rell\) be closed with empty interior. Then, there exists a differentiable and locally Lipschitz function \(\bar{g}\colon 
    \Rell \to \R\) such that \(\bar{g}(z)= 0\) for every \(z \in S\), and \(\bar{g}(z) > 0\) for every \(z \in \Rell \setminus S\). Moreover  \(\bar{g}\) is continuously differentiable everywhere except on \(S\) with \(\Bcal_c(0,1) \subset \clarke \bar{g}(z)\) for any \(z \in S\). In particular, if \(S\) is defined as in Section~\myref{sect:failure}, \(\bar{g}\) is also path differentiable.
\end{claim}
\begin{proof} 
    Consider the bump function \(\bar{b}\colon \R \to \R\) defined by
     \[
        \bar{b}(x) \coloneq
        \begin{cases}
         e^{-1/(1-x^2)} & \text{if } |x| < 1, \\
        0 & \text{if } |x| \geq 1.
        \end{cases}
    \]
    Let \(b(x) = \left(\bar{b}(x)\right)^2\). The derivative \(b^{'}\) is odd and \(b^{'}(1/2) < 0\). Define \(\varphi\colon \R \to \R\) as \(\varphi(x) = \frac{1}{-b^{'}(1/2)} b(x)\).
    The function \(\varphi\) is infinitely differentiable and is support is \([-1,1]\). In addition, \(\varphi^{\prime}(1/2) = -1\) and \(\varphi^{\prime}(-1/2) = -\varphi^{\prime}(1/2) = 1 \). 

    By Whitney theorem \cite[Theorem 2.29]{Lee2012ISM}, \(S\) is the zero of an infinitely differentiable and non-negative function \(\bar{f}\colon \Rell \to \R\) and, for any \(z \in S\), \(\nabla \bar{f}(z) = 0\). Thanks to Lemma~\myref{lem:subseqtocantor}, there exists a sequence \((z_k)_{k\in\N} \subset \Rell \setminus S\) such that its set of accumulation points is \(S\). For any \(k \in \N\), consider \(\Bcal(z_k,r_k)\) such that \(r_k > 0\), \(r_k \xrightarrow[k \to \infny]{} 0\), and \(\Bcal(z_k,r_k) \cap \Bcal(z_{s},r_{s}) = \varnothing\) for \(s \neq k\). Such a choice is possible, because for each $k$, $z_k$ is at positive distance of \(z_i\) for each $i \neq k$ and at positive distance of \(S\) the set of accumulation points.
    Define \(\bar{g}\colon \Rell \to \R\) as 
    \[
    z \mapsto \bar{f}(\vz) + \sum_{k \in \N} \varphi\left(\frac{\norm{\vz - \vz_k}}{r_k}\right)r_k. 
    \]
    Then, \(\bar{g}(z)= 0\) for every \(z \in S\), and \(\bar{g}(z) > 0\) for every \(z \in \Rell \setminus S\). Moreover  \(\bar{g}\) is differentiable and locally Lipschitz on \(\Rell\). It is continuously differentiable everywhere except on \(S\) because \(\Bcal_c(0,1) \subset \clarke \bar{g}(z)\) for any \(z \in S\). Indeed, let \(\bar{z} \in S\) and, without loss of generality, we assume that \(z_k \xrightarrow[k \to \infny]{} \bar{z}\). Let \(v \in \{w \in \Rell: \norm{w} = 1\}\). Consider \((\bar{z}_k)_{k\in\N}\) such that \(\bar{z}_k = z_k - \frac{r_k}{2}v\) for all \(k\in\N\). Then we have, for any \(k \in \N\), \(\nabla \bar{g}(\bar{z}_k) = \nabla \bar{f}(\bar{z_k}) + v\). So, \(\bar{z}_k \xrightarrow[k \to \infny]{} \bar{z}\) and \(\nabla\bar{g}(\bar{z}_k)\xrightarrow[k \to \infny]{} v\). It follows that \(v \in \clarke \bar{g}(\bar{z})\). Since \(v\) is arbitrary in \(\{w \in \Rell: \norm{w} = 1\}\), and by Definition~\myref{def:clarkesubdiff},
    we have 
    \[ \conv\,\{w \in \Rell: \norm{w} = 1\} = \Bcal_c(0,1) \subset \clarke \bar{g}(\bar{z}).\]
    Now let us assume that \(S\) is defined as in Section~\myref{sect:failure}. Let \(\gamma\colon [0,1] \to \Rell\) be an absolutely continuous curve. Given that \(\bar{g}\) is locally Lipschitz, \(\bar{g} \circ \gamma \) is absolutely continuous. Let \(\Omega \subset [0,1]\) be the full measure set where \(\bar{g} \circ \gamma\) and \(\gamma\) are differentiable.
    Consider the following three sets:
    \begin{align}
         &S_1 = \{t \in \Omega: \gamma(t) \in \Rell \setminus S\}, \nonumber \\
        &S_2 = \{t \in \Omega: \gamma(t) \in S \text{ and } \dot{\gamma}(t) = 0\}, \nonumber \\
        &S_3 = \{t \in \Omega: \gamma(t) \in S \text{ and } \dot{\gamma}(t) \neq 0\}. \nonumber
    \end{align}
    We have \(\Omega = S_1 \cup S_2 \cup S_3\). Since \(\bar{g}\) is continuously differentiable outside \(S\), the chain rule holds on \(S_1\), i.e., for any \(t \in S_1\), 
    \begin{align*}
    \frac{d}{dt}\bar{g}(\gamma(t)) = \scalarp{u,\dot{\gamma}(t)}\quad \forall\, u \in \clarke \bar{g}(\gamma(t)) = \{\nabla \bar{g}(\gamma(t))\}.
    \end{align*}
    Thanks to \cite[Lemma 6]{pauwels2023conservative}, we know that \(S_3\) has zero measure. As \(\bar{g}\) is locally Lipschitz and \(\gamma(t+h) = \gamma(t) + o(h)\) for all \((t+h,t) \in \Omega\times S_2\) with \(|h|\) small enough, the chain rule holds in \(S_2\). 
    In conclusion, the chain rule holds for almost all \(t \in [0,1]\) and \(\bar{g}\) is path differentiable.
\end{proof}
%%%%%%%%%%%%%%%%%%%%%%%%%%%%%%%%%%%%%%%%%%%%%%%%%%%%%%%%%%%%%%%%%%%%%%%%%%%%%%%%%%%
%%%%%%%%%%%%%%%%%%%%%%%%%%%%%%%%%%%% Bibliography printing %%%%%%%%%%%%%%%%%%%%%%%%%%%%%
%%%%%%%%%%%%%%%%%%%%%%%%%%%%%%%%%%%%%%%%%%%%%%%%%%%%%%%%%%%%%%%%%%%%%%%%%%%%%%%%%%%

\runinsectionstar[\large]{Acknowledgements}
The authors warmly thank Bruno Despr\'es and François Pacaud for relevant discussions and suggestions. 

This work has been supported by the Occitanie region, the European Regional Development Fund (ERDF), and the French government, through the France 2030 project managed by the National Research Agency (ANR) with the reference number ``ANR-22-EXES-0015''.

JB and EP thank AI Interdisciplinary Institute ANITI funding, through the ANR under the France 2030 program (grant ANR-23-IACL-0002), Chair TRIAL, Air Force Office of Scientific Research, Air Force Material Command, USAF, under grant numbers FA8655-22-1-7012. JB and EP acknowledge support from ANR MAD. JB and EP thank TSE-P and acknowledge support from ANR Chess, grant ANR-17-EURE-0010, ANR Regulia. EP acknowledges support from IUF. As part of the Madlearning project, this work received funding from the French government, managed by the National Research Agency (ANR) under the France 2030 program, with reference number 'ANR-25-PEIA-0002'.

\bibliographystyle{plainnat}%
\bibliography{references}
%%%-------------------------------------------------
\end{document}